\documentclass[11pt,twoside]{article}

\usepackage{amsbsy,amsfonts,amsmath,amssymb,eucal,mathrsfs,graphicx,manfnt}
\usepackage[all,cmtip]{xy}
\usepackage[T1]{fontenc}
\usepackage{tikz}
\usetikzlibrary{arrows,matrix,positioning}
\usepackage{tikz-cd}
\usepackage{hyperref}
\let\svthefootnote\thefootnote
\newcommand\blankfootnote[1]{%
  \let\thefootnote\relax\footnotetext{#1}%
  \let\thefootnote\svthefootnote%
}

\def\no{\noindent}

\renewcommand{\ge}{\geqslant}
\renewcommand{\le}{\leqslant}
\def\itemn#1{\item[\hspace{0.6mm} {\rm (#1)}]}
\def\itemm#1{\item[\hspace{10mm} {\rm (#1)}]}
\def\To{\Rightarrow}
\def\too{\longrightarrow}
\def\tooo{\relbar\joinrel\relbar\joinrel\longrightarrow}

\def\into{\hookrightarrow}

\def\isomto{\xrightarrow{\,\smash{\raisebox{-0.5ex}{\ensuremath{\scriptstyle\sim}}}\,}}

\mathchardef\ordinarycolon\mathcode`\:
\mathcode`\:=\string"8000
\begingroup \catcode`\:=\active
  \gdef:{\mathrel{\mathop\ordinarycolon}}
\endgroup

\newcommand{\spref}[1]{\href{http://stacks.math.columbia.edu/tag/#1}{Tag~#1}}

\newtheorem{counter}[subsubsection]{$\!\!$}
\newenvironment{definition}
{\begin{counter} \rm {\bf Definition.}}{\end{counter}}
\newenvironment{definitions}
{\begin{counter} \rm {\bf Definitions.}}{\end{counter}}

\newenvironment{prop}{\begin{counter} {\bf Proposition.}}{\end{counter}}
\newenvironment{lemma}{\begin{counter} {\bf Lemma.}}{\end{counter}}

\newenvironment{corollary}{\begin{counter} {\bf Corollary.}}{\end{counter}}

\newenvironment{theorem}{\begin{counter} {\bf Theorem.}}{\end{counter}}

\newenvironment{remark}{\begin{counter} \rm {\bf Remark.}}{\end{counter}}
\newenvironment{remarks}{\begin{counter} \rm {\bf Remarks.}}{\end{counter}}
\newenvironment{exam}{\begin{counter} \rm {\bf Example.}}{\end{counter}}
\newenvironment{examples}{\begin{counter} \rm {\bf Examples.}}{\end{counter}}

\newenvironment{thematic-item}{\begin{counter} \rm}{\end{counter}}
\newenvironment{proof}{{\flushleft \bf Proof~:}}{\hfill $\square$ \vspace{5mm}}

\newtheorem{subcounter}[subsection]{$\!\!$}
\newenvironment{definition*}{\begin{subcounter} \rm {\bf Definition.}}{\end{subcounter}}
\newenvironment{defis*}{\begin{subcounter} \rm {\bf Definitions.}}{\end{subcounter}}
\newenvironment{nota*}{\begin{subcounter} \rm {\bf Notation.}}{\end{subcounter}}
\newenvironment{situ*}{\begin{subcounter} \rm {\bf Situation.}}{\end{subcounter}}
\newenvironment{prop*}{\begin{subcounter} {\bf Proposition.}}{\end{subcounter}}
\newenvironment{lemma*}{\begin{subcounter} {\bf Lemma.}}{\end{subcounter}}
\newenvironment{fact*}{\begin{subcounter} {\bf Fact.}}{\end{subcounter}}
\newenvironment{coro*}{\begin{subcounter} {\bf Corollary.}}{\end{subcounter}}
\newenvironment{conj*}{\begin{subcounter} {\bf Conjecture.}}{\end{subcounter}}
\newenvironment{theorem*}{\begin{subcounter} {\bf Theorem.}}{\end{subcounter}}
\newenvironment{quot*}{\begin{subcounter} {\bf Theorem \ }}{\end{subcounter}}
\newenvironment{cons*}{\begin{subcounter} \rm {\bf Construction.}}{\end{subcounter}}
\newenvironment{remark*}{\begin{subcounter} \rm {\bf Remark.}}{\end{subcounter}}
\newenvironment{remarks*}{\begin{subcounter} \rm {\bf Remarks.}}{\end{subcounter}}
\newenvironment{exam*}{\begin{subcounter} \rm {\bf Example.}}{\end{subcounter}}
\newenvironment{exams*}{\begin{subcounter} \rm {\bf Examples.}}{\end{subcounter}}
\newenvironment{ques*}{\begin{subcounter} \rm {\bf Question.}}{\end{subcounter}}
\newenvironment{cont*}{\begin{subcounter} \rm {\bf Counter-examples.}}{\end{subcounter}}
\newenvironment{todo*}{\begin{counter} \rm {\bf TODO.}}{\end{counter}}
\newenvironment{noth*}{\begin{subcounter} \rm}{\end{subcounter}}
\newenvironment{rien*}{\begin{subcounter} \rm}{\end{subcounter}}

\DeclareMathOperator{\Aut}{Aut}
\DeclareMathOperator{\Hom}{Hom}
\DeclareMathOperator{\Isom}{Isom}

\DeclareMathOperator{\AlgSp}{AlgSp}
\DeclareMathOperator{\Pic}{Pic}

\DeclareMathOperator{\QCoh}{QCoh}

\DeclareMathOperator{\id}{id}

\DeclareMathOperator{\Spec}{Spec}

\DeclareMathOperator{\im}{im}

\DeclareMathOperator{\fppf}{fppf}

\DeclareMathOperator{\pr}{pr}

\DeclareMathOperator{\Stab}{Stab}

\DeclareMathOperator{\ev}{ev}
\DeclareMathOperator{\interp}{int}

\DeclareMathOperator{\card}{card}
\DeclareMathOperator{\Flat}{Flat}
\DeclareMathOperator{\qa}{qa}
\DeclareMathOperator{\post}{post}
\DeclareMathOperator{\pre}{pre}

\def\cA{{\mathcal A}} \def\cB{{\mathcal B}} \def\cC{{\mathcal C}}
  \def\cF{{\mathcal F}}
\def\cG{{\mathcal G}} \def\cH{{\mathcal H}} \def\cI{{\mathcal I}}
  
  \def\cO{{\mathcal O}}
  
  \def\cX{{\mathcal X}}
  \def\cX{{\mathcal X}}

\renewcommand\AA{\mathbb{A}}

\newcommand\GG{\mathbb{G}}

 \newcommand\NN{\mathbb{N}}

 \newcommand\ZZ{\mathbb{Z}}

 \def\sB{\mathscr{B}} 
  
\def\sG{\mathscr{G}}  
  \def\sL{\mathscr{L}}
  \def\sO{\mathscr{O}}
\def\sP{\mathscr{P}}  \def\sR{\mathscr{R}}
 \def\sT{\mathscr{T}} \def\sU{\mathscr{U}}
  \def\sX{\mathscr{X}}

\newlength{\myarrowsize} 
    \newlength{\myoldlinewidth}
\pgfarrowsdeclare{myto}{myto}{
        \pgfsetdash{}{0pt} 
        \pgfsetbeveljoin 
        \pgfsetroundcap 
        \setlength{\myarrowsize}{0.6pt}
        \addtolength{\myarrowsize}{.5\pgflinewidth}
        \pgfarrowsleftextend{-4\myarrowsize-.5\pgflinewidth} 
        \pgfarrowsrightextend{.7\pgflinewidth}
    }{
        \setlength{\myarrowsize}{0.6pt} 
        \addtolength{\myarrowsize}{.5\pgflinewidth}  
        \setlength{\myoldlinewidth}{\pgflinewidth}
        \pgfsetroundjoin
        \pgfsetlinewidth{0.0001pt}
        \pgfpathmoveto{\pgfpoint{0.43\myarrowsize}{0}}
        \pgfpatharc{0}{70}{0.14\myarrowsize}
        \pgfpatharc{-110}{-169.5}{4\myarrowsize}
        \pgfpatharc{10.5}{189}{0.25\myarrowsize and 0.12\myarrowsize}
        \pgfpatharc{-170}{-119.5}{4.48\myarrowsize}
        \pgfpathmoveto{\pgfpoint{0.43\myarrowsize}{0}}
        \pgfpatharc{0}{-70}{0.14\myarrowsize}
        \pgfpatharc{110}{169.5}{4\myarrowsize}
        \pgfpatharc{-10.5}{-189}{0.25\myarrowsize and 0.12\myarrowsize}
        \pgfpatharc{170}{119.5}{4.48\myarrowsize}
        \pgfpathclose
        \pgfsetstrokeopacity{0.25}
        \pgfusepathqfillstroke
    }

\begin{document}

\begin{center}
{\Large \bf The complexity of a flat groupoid}
\end{center}

\begin{center}
Matthieu Romagny, David Rydh\footnote{The second author was supported by the Swedish Research Council 2015-05554 and the G\"oran Gustafsson foundation.} and Gabriel Zalamansky
\end{center}

\begin{center}
{\sc Abstract}
\end{center}

\begin{minipage}{15cm}
Grothendieck proved that any finite epimorphism of noetherian
schemes factors into a finite sequence of effective epimorphisms.
We define the complexity of a flat groupoid $R\rightrightarrows X$ with finite
stabilizer to be the length of the canonical sequence of the finite
map $R\to X\times_{X/R} X$, where $X/R$ is the Keel--Mori geometric quotient. For groupoids of complexity at most 1, we prove a
theorem of descent along the quotient $X\to X/R$ and a theorem
on the existence of the quotient of a groupoid by a normal subgroupoid. We expect that
the complexity could play an important role in the finer study
of quotients by groupoids.
\end{minipage}

\blankfootnote{\hspace{-6.8mm}
Date : April 26, 2018 \\
Keywords~: Groupoids, group schemes, quotients, algebraic spaces, effective epimorphisms, descent \\
Mathematics Subject Classification: 14A20,14L15,14L30}

%{\def\thefootnote{\relax}
%\footnote{ \hspace{-6.8mm}
%Keywords~: Groupoids, group schemes, quotients, algebraic spaces, effective epimorphisms, descent \\
%Mathematics Subject Classification: 14A20,14L15,14L30 \\
%}}

\tableofcontents

\section{Introduction}

{\bf Motivation.}
Let $X$ be a scheme endowed with an action of a group scheme
$G$ such that there exists a quotient $\pi:X\to Y=X/G$.
Consider the category $\cC(X)$ of vector bundles on~$X$.
In this paper,
we give new examples where one can describe
which $G$-linearized bundles on $X$ that descend to bundles
on~$Y$, and similarly for other fibered categories $\cC$.
More precisely, let
$\cC(G,X)$ be the category of vector bundles
endowed with a $G$-linearization.
%For each geometric point $x:\Spec(k)\to X$ let $G_x$ be its
%stabilizer, and
Let $\cC(G,X)'$ be the subcategory of $G$-linearized bundles for
which the action of the stabilizers of geometric points is
trivial. It is not hard to see that for any vector bundle
$\cG\in\cC(Y)$, the pullback $\cF=\pi^*\cG$ is naturally an
object of $\cC(G,X)'$. The question is:

\bigskip

Let $G\times X\to X$ be a group scheme
action as above, with quotient $\pi:X\to Y=X/G$. When is
the pullback $\pi^*:\cC(Y)\to \cC(G,X)'$ an equivalence?

\bigskip

The correct framework for this type of question is that of
algebraic spaces (which generalize schemes) and groupoids
(which generalize group actions). That this is so was
demonstrated twenty years ago by Keel and Mori who settled
the question of existence of quotients for actions with
finite stabilizer in the paper~\cite{KM97}. The main point
is that
groupoids allow reduction and d\'evissage in a much more
flexible way than group actions. Moreover, groupoids include
examples of interest like foliations in characteristic $p$,
and inseparable equivalence relations as in work of Rudakov
and Shafarevich~\cite{RS76} and Ekedahl~\cite{Ek88},
which we will return to in the end of this introduction. We
emphasize that our results are equally interesting in the
restricted case of group actions. So in the sequel we let
\begin{trivlist}
\itemn{1} $R\rightrightarrows X$ be a flat locally finitely presented groupoid of
algebraic spaces,
\itemn{2} $\cC\to \AlgSp$ be a category fibered over the category
of algebraic spaces,
\itemn{3} $\cC(R,X)$ be the category of objects of $\cC(X)$
equipped with $R$-linearizations (see~\ref{ss:equivariant objects} for a precise definition), and
\itemn{4} $\cC(R,X)'\subset \cC(R,X)$ be the full subcategory of objects with trivial
geometric stabilizer actions.
\end{trivlist}
Since $R$-linearized objects on $X$ are the same as objects
on the algebraic stack $\sX=[X/R]$, the language of stacks
is an alternative which is also used on that matter.

\smallskip

\noindent {\bf Known results.}
When $X\to Y$ is a {\em tame} quotient, which means that
the geometric stabilizers of $R\rightrightarrows X$ are linearly
reductive finite group schemes, and~$\cC$ is either
the category of line bundles, or finite \'etale covers,
or torsors under a fixed linearly reductive finite group
scheme, Olsson showed that $\pi^*:\cC(Y)\to \cC(G,X)'$
is an equivalence \cite[Props.~6.1, 6.2, 6.4]{Ol12}. When
$X\to Y$ is a {\em good} quotient and $\cC$ is the category
of vector bundles, Alper showed that $\pi^*$ is an
equivalence~\cite[Thm.~10.3]{Al13}. Results for
good quotients and other categories~$\cC$ will be presented
in an upcoming paper by the second author.

\smallskip

\noindent {\bf The complexity.}
We wish to find examples that go beyond these cases, e.g.,
wild actions in characteristic $p$. In this new setting
the map $\pi^*:\cC(Y)\to \cC(G,X)'$ fails to be an isomorphism
in general; e.g. if $\cC$ is the category of line bundles
and $X=\Spec( k[\epsilon]/(\epsilon^2) )$ with trivial action of
$G=\ZZ/p\ZZ$, the $G$-line bundle $L$ generated by a section
$x$ with action $x\mapsto (1+ \epsilon)x$ is not trivial.
For this, we introduce a new invariant of flat groupoids which
we call the {\em complexity}. (This is not to be confused with
the complexity as defined by Vinberg~\cite{Vi86} in another
context, namely the minimal codimension of a Borel orbit
in a variety acted on by a connected reductive group.)
We fix our attention on the morphism
$j_Y:R\to X\times_Y X$ which is finite and surjective when the groupoid
has finite inertia.
The complexity of the groupoid is controlled by the
epimorphicity properties of this map. In order to quantify
this, we use a result of Grothendieck to the effect that a
finite epimorphism of noetherian schemes factors as
a finite sequence of {\em effective} epimorphisms. We prove
in~\ref{prop:canonical_factorization} that there is a canonical
such sequence, and we define the complexity of $R\rightrightarrows X$ as the
length of the canonical sequence of~$j_Y$.

\smallskip

\noindent {\bf Main new results.}
The complexity is equal to $0$ when $j_Y$ is an isomorphism,
which means that the groupoid acts freely; in this case most
questions involving $R\rightrightarrows X$ are easily answered.
The next case in difficulty is the case of complexity $1$.
In order to obtain results in this case, we introduce the
stabilizer $\Sigma$ of $R\rightrightarrows X$, which is the preimage of
the diagonal under $R\to X\times X$. It refines the information
given by the collection of stabilizers of geometric points in
that it accounts for higher ramification. We let
$\cC(R,X)^\Sigma\subset\cC(R,X)'$ be the subcategory
of $R$-linearized objects for which the action of $\Sigma$ is
trivial. In our main result
we have to assume that the quotient map is flat; the payoff
is that we can handle very general categories $\cC$.

\bigskip

{\bf Theorem~\ref{theorem:descent_along_quotient}.}
{\em Let $R\rightrightarrows X$ be a flat, locally finitely presented groupoid space
with finite stabilizer $\Sigma\to X$ and complexity at most $1$.
Assume that the quotient $\pi:X\to Y=X/R$ is flat (resp.\ flat
and locally of finite presentation). Let $\cC\to\AlgSp$ be a stack in
categories for the fpqc topology (resp.\ for the fppf topology).
\begin{trivlist}
\itemn{1} If the sheaves of homomorphisms $\cH om_{\cC}(\cF,\cG)$
have diagonals which are representable by algebraic spaces, then
the pullback functor $\pi^*:\cC(Y)\to \cC(R,X)^{\Sigma}$ is fully faithful.
\itemn{2} If the sheaves of isomorphisms $\cI som_{\cC}(\cF,\cG)$ 
are representable by algebraic spaces, then the pullback functor
$\pi^*:\cC(Y)\to \cC(R,X)^{\Sigma}$ is essentially surjective.
\end{trivlist}
In particular if $\cC$ is a stack in groupoids with representable
diagonal, the functor $\pi^*$ is an equivalence.}

\bigskip

This applies to stacks whose diagonal has some representability
properties. The next theorem applies to a stack which does not
enjoy such a property.

\bigskip

{\bf Theorem~\ref{theorem:descent_of_flat}.}
{\em Let $\cC\to\AlgSp$ be the fppf stack in categories whose objects over $X$
are flat morphisms of algebraic spaces $X'\to X$.
Let $R\rightrightarrows X$ be a flat, locally finitely presented groupoid space
with finite stabilizer $\Sigma\to X$ and complexity at most $1$.
Assume that the quotient $\pi:X\to Y=X/R$ is flat and locally
of finite presentation. Then the functor
$\pi^*:\cC(Y)\to \cC(R,X)^{\Sigma}$ is an equivalence.}

\bigskip

We give examples of groupoids satisfying the assumptions of
these theorems in
section~\ref{examples}. These include groupoids acting on
smooth schemes in such a way that the stabilizers are symmetric
groups acting by permutation of local coordinates. Other
examples are given by groupoids acting on curves in positive
characteristic; this is especially interesting in
characteristic~2. The two theorems above can fail when $\pi$ is not
flat and the stabilizer groups are not tame, see section~\ref{non-flat-example}.
We do not know if the assumption that the complexity is at most one is
necessary.

We give an application to the existence
of quotients of groupoids by normal subgroupoids. This is
interesting when applying d\'evissage arguments, as for
instance in \cite[\S~7]{KM97}. This question is also natural
from the point of view of understanding the internal structure
of the category of groupoids. The basic observation is this:
if $R\rightrightarrows X$ is a groupoid $P\subset R\rightrightarrows X$ is a normal
flat subgroupoid, the actions of $P$ on $R$ by precomposition
and postcomposition are free, but the simultaneous action
of $P\times P$ is {\em not} free. For groupoids $R=G\times X$
given by group actions, it is nevertheless easy to make
$G/H$ act on $X/H$, providing a quotient groupoid
$G/H\times X/H\rightrightarrows X/H$. However for general groupoids,
constructing a composition law on the quotient
$P\backslash R/P$ making it a groupoid acting on $X/P$
is much more complicated. In
section~\ref{section:quotient_by_gpd} we review some cases
where this is possible.
For subgroupoids of complexity 1
with flat quotient, we obtain a satisfying answer.

\bigskip

{\bf Theorem~\ref{theorem:quotient_by_subgroupoid}.}
{\em Let $R\rightrightarrows X$ be a flat, locally finitely presented groupoid
of algebraic spaces.
Let $P\rightrightarrows X$ be a flat, locally finitely presented normal
subgroupoid of $R$ with finite stabilizer $\Sigma_P\to X$
and complexity at most $1$. Assume that the quotient
$X\to Y=X/P$ is flat and locally finitely presented. Then there is a quotient
groupoid $Q\rightrightarrows Y$ which is flat and locally finitely presented,
with $Q=P\backslash R/P$. Moreover, the morphisms
$R\to Q$ and $R\times_X R\to Q\times_YQ$ are flat and locally finitely presented.}

\bigskip

\no {\bf Directions of further work.}
The natural question now is to extend these results
to the case of groupoids of complexity 2. This would most
likely shed some light on the case of arbitrary complexity.
For the moment, we have no idea of what the correct
substitute for $\cC(R,X)^{\Sigma}$ should be in the general
context.

The application we envision for these results is to the
study of finite flat covers of algebraic varieties,
typically over a field $k$ of characteristic $p$. More
precisely, we expect our theorems to be useful for
understanding how purely inseparable morphisms of algebraic
$k$-varieties $f:V\to W$ can be factorized. An important
instance is when $f$ is an iterate of the Frobenius morphism
of~$V$. We note that when $V$ is smooth, $f$ will be
flat. Thus the assumption of flatness of the quotient
map in our results is not too annoying; we give some more
comment on this point in Remark~\ref{remark:philosophy}.

\smallskip

\noindent {\bf Organization of the article}. As we said already,
we work in the setting of groupoids in algebraic spaces.
(The relevance of this choice in questions of quotients in
Algebraic Geometry is well explained in the paper \cite{Li05}
which we recommend as an excellent contextual reading.)
This leads us to start in section~\ref{section:finite_epis}
with some preparations on finite epimorphisms of spaces.
In particular, we give sufficient conditions for an epimorphism
of algebraic spaces to be effective, and we prove a precise
form of Grothendieck's factorization of finite epimorphisms
into finite effective epimorphisms. In
section~\ref{section:groupoids} we recall
the basic vocabulary of groupoids, we define the complexity,
and we present several examples. Finally in
section~\ref{section:main theorems} we prove the main
results of the paper, presented above.

\smallskip

\noindent {\bf Acknowledgements.}
This article is derived from the third author's Ph.D.\ thesis.
We thank user27920 on MathOverflow
for help in the proof of
Proposition~\ref{prop:canonical_factorization} before we learned
this is in \cite{SGA6}. We thank Alessandro Chiodo for discussions
related to Theorem~\ref{theorem:descent_along_quotient}. We
thank C\'edric Bonnaf\'e for his interest and for discussions
around actions of groups generated by reflections.

\section{Finite epimorphisms} \label{section:finite_epis}

This section of preliminary nature contains material on finite
epimorphisms of algebraic spaces.
The notion of epimorphism turns out to be a little more subtle
in the category of algebraic spaces than its counterpart in the
category of schemes, due to the lack of the locally ringed space
description. The same is true for the notion of effective epimorphism.
In order to have a better understanding of the situation, we will
give some manageable conditions that ensure that a map of algebraic
spaces is an epimorphism, or an effective epimorphism.
The main result is Theorem~\ref{theo:eff_epi_of_spaces}, but
for the convenience of the reader we will indicate here its
main consequence needed in the sequel. Recall the following two
statements in the easy scheme case:

\begin{prop}
Let $f:S'\to S$ be a qcqs surjective morphism of schemes.
Write $\cA(S')=f_*\cO_{S'}$.
Then the following are equivalent:
\begin{trivlist}
\itemn{1} $f$ is schematically dominant, that is,
$\cA(S)\to \cA(S')$ is injective;
\itemn{2} $f$ is an epimorphism in the category of schemes.
\end{trivlist}
\end{prop}

\begin{prop}
Let $f:S'\to S$ be a qcqs submersive morphism of schemes.
Write $S''=S'\times_S S'$.
Then the following are equivalent:
\begin{trivlist}
\itemn{1} $\cA(S)\to\cA(S')\rightrightarrows \cA(S'')$ is exact;
\itemn{2} $f$ is an effective epimorphism in the category of schemes.
\end{trivlist}
\end{prop}

The main results we shall need are the following:

\begin{prop}
(Lemma~\ref{lemm: e-submersion is epi})
Let $f:S'\to S$ be a qcqs morphism of algebraic spaces which is
submersive after every \'etale base change on $S$.
Then the following are equivalent:
\begin{trivlist}
\itemn{1} $f$ is schematically dominant, that is, $\cA(S)\to\cA(S')$ is injective;
\itemn{2} $f$ is an epimorphism in the category of algebraic spaces.
\end{trivlist}
\end{prop}

\begin{prop}
(Lemma~\ref{lemma:effective_epi_and_functions} +
Corollary~\ref{coro:uniform_effective_epi})
Let $f:S'\to S$ be an integral morphism of algebraic spaces.
Then the following are equivalent:
\begin{trivlist}
\itemn{1} $\cA(S)\to \cA(S')\rightrightarrows \cA(S'')$ is exact;
\itemn{2} $f$ is an effective epimorphism in the category of algebraic spaces.
\end{trivlist}
Under these equivalent conditions, $f$ is a uniform effective epimorphism.
\end{prop}

Finally we prove Grothendieck's
factorization of a finite epimorphism into a finite sequence of
finite effective epimorphisms,
Proposition~\ref{prop:canonical_factorization}, placing ourselves
in a slightly more general context and giving some useful complements.

\subsection{Epimorphisms}

First we recall an easy characterization of epimorphisms of schemes.

\begin{lemma} \label{lemma:epi_of_schemes}
Let $f:S'\to S$ be a morphism of schemes. The following
conditions are equivalent:
\begin{trivlist}
\itemn{1} $f$ is an epimorphism (of schemes).
\itemn{2} $f$ does not factor through an open or closed subscheme $Z\subsetneq S$.
\itemn{3} $f$ does not factor through a subscheme $Z\subsetneq S$.
\end{trivlist}
\end{lemma}

\begin{proof}
(1) $\To$ (2). Assume that $f$ factors through a subscheme
$Z\subsetneq S$ which is either open or closed. Let $X=S\amalg_Z S$ be
the ringed space obtained by gluing two copies of $S$ along their common
copy of~$Z$. If $Z$ is open then $X$ is a scheme by ordinary topological
gluing, and if $Z$ is closed then $X$ is a scheme by
Ferrand~\cite[Thm.~7.1]{Fe03} or \cite[\spref{0B7M}]{SP}.
Let $u,v:S\to X$ be the canonical maps.
We have $u\ne v$ and $uf=vf$, so~$f$ is not an epimorphism.

\smallskip

\no (2) $\To$ (3) Immediate because a subscheme is a closed subscheme
of an open subscheme.

\smallskip

\no (3) $\To$ (1). Let $X$ be a scheme and let $u,v:S\to X$ be morphisms
such that $uf=vf$. Let $Z$ be the preimage of the diagonal
$\Delta:X\to X\times X$ by the map $(u,v):S\to X\times X$. Since $\Delta$
is an immersion, then~$Z$ is a subscheme of $S$. Since $f$ factors through $Z$,
by (3) it follows that $Z=S$. This shows that $(u,v)$ factors through the
diagonal, that is $u=v$.
\end{proof}

Recall that an algebraic space is called {\em locally separated}
if its diagonal is an immersion. Clearly the lemma and its proof show
that an epimorphism of schemes is also an epimorphism in the category
of locally separated algebraic spaces. However, it may fail to
be an epimorphism in the category of all algebraic spaces, even if it is
surjective and schematically dominant. Here is a counter-example.

\begin{exam}
Let $k$ be a field of characteristic $\ne 2$.
Consider the scheme
\[
S=\Spec(k[x,y]/(x^2-y^2))
\]
with closed
subscheme $Y=V(x-y)$ and open complement $U=D(x-y)=S\setminus Y$.
Let $S'=Y\amalg U$. Then the canonical map $f:S'\to S$ is a surjection
to a reduced scheme, hence an epimorphism of schemes by the lemma above.
The map $j:S'\to S\subset \AA^1_k\times\AA^1_k$ defines an
\'etale equivalence relation on $\AA^1_k$. We let $\pi:\AA^1_k\to X$
be the quotient algebraic space. By construction, the pullback of the diagonal $X\subset X\times X$ to $\AA^1_k\times\AA^1_k$ is $S'$.
Let $u,v:S\to\AA^1_k\to X$ be the maps induced by the two projections
$\pr_1,\pr_2:S\to \AA^1_k$. These maps are distinct, since
otherwise $(u,v)$ would factor through the diagonal of $X$, which would mean
that $(p_1,p_2):S\to \AA^1_k\times\AA^1_k$ factors through $S'$, which it does
not. However $uf=vf$, hence $f$ is not an epimorphism of algebraic spaces.
\end{exam}

In the applications that we have in mind, it is cumbersome to check
that the algebraic spaces involved satisfy some separation condition.
Because of this, we spend some effort on obtaining criteria for
epimorphisms in the category of all algebraic spaces. In order to put
\ref{lemma:epi_of_schemes} in perspective, it is useful
to have the construction of gluing along closed subschemes
available for algebraic spaces. This is originally due to Raoult \cite{Ra74}.
Variants appear in \cite[Thm.~6.1]{Ar70}, \cite[Thm.~A.4]{Ry11},
\cite[Thm.~2.2.2]{CLO12}, \cite[Thm.~5.3.1]{TT16}. In all these sources, the
hypotheses allow one of the maps $f,g$ of the gluing diagram to be finite or
at least affine and usually some noetherian-like assumptions are present.
It is known to most people that these assumptions are not essential
at least when both maps $f,g$ are closed immersions; we give a statement
with the main input for the proof coming from \cite{SP}.

\begin{lemma} \label{lemm:gluing_along_closed_for_alg_spaces}
Let $i_1:Y\into X_1$ and $i_2:Y\into X_2$ be closed immersions of algebraic
spaces. Then, there exists a pushout $W=X_1\amalg_Y X_2$ in the category
of algebraic spaces:
\[
\xymatrix@C=15mm{
Y \ar[r]^{i_2} \ar[d]_{i_1} & X_2 \ar[d]^b \\
X_1 \ar[r]^-a & W.}
\]
Moreover, the diagram is a cartesian square; the maps $a,b$ are closed
immersions; the pushout is topological, i.e., its underlying topological
space is $|X_1|\amalg_{|Y|} |X_2|$; and there is a short exact sequence
\[
0 \too \cO_W \too a_*\cO_{X_1}\oplus b_*\cO_{X_2} \too
c_*\cO_Y\too 0
\]
of sheaves on the small \'etale site of $W$.
\end{lemma}

\begin{proof}
We will reduce to the known case of schemes. For this
we will use the following classical extension result for
\'etale maps: if $U,E,E'$ are disjoint unions
of affine schemes (henceforth to be called {\em sums of affines}
for brevity) and $E\into U$ is a closed immersion,
and $E'\to E$ is an \'etale morphism, then there exists
a sum of affines~$U'$ and an \'etale
morphism $U'\to U$ such that $E'\simeq U'\times_U E$.
The proof can be found for example in \cite[\spref{04D1}]{SP}.
Note that if $E'\to E$ is surjective, we may choose
$U'\to U$ surjective by adding to $U'$ the sum of affines
in a Zariski covering of $U\setminus E$.

For each $i=1,2$ let $\pi_i:U_i\to X_i$ be an \'etale surjective map
where~$U_i$ is a sum of affines. Let $E_i=U_i\times_{X_i}Y$. Then
$E_1\times_Y E_2$ is \'etale surjective over $E_1$ and $E_2$.
Let $E'$ be the sum of affines given by a Zariski covering of
$E_1\times_Y E_2$. By the fact quoted above, for each $i=1,2$ there
exists $U'_i\to U_i$ \'etale surjective whose restriction to $E_i$
is isomorphic to $E'$. In this way, replacing $U_i$ by $U'_i$ we
see that we can assume that $E_1\simeq E_2$. Now for $i=1,2$ let
$R_i=U_i\times_{X_i} U_i$ with its two projections $s_i,t_i:R_i\to U_i$.
Let $F_i$ be the preimage of $Y$ in $R_i$. Since $\pi_is_i=\pi_it_i$,
this is isomorphic to the preimage of $E_i$ under any of the maps
$s_i$ or $t_i$. The isomorphism $E_1\simeq E_2$ induces a compatible
isomorphism $F_1\simeq F_2$; in the sequel we view these isomorphisms
as identifications so we write $E=E_1=E_2$ and $F=F_1=F_2$.

By the scheme
case the pushouts $\sU:=U_1\amalg_E U_2$ and $\sR:=R_1\amalg_F R_2$
make sense as schemes. Using the pushout property for $\sR$ we see
that the maps $s\amalg s,t\amalg t: R_1\amalg R_2\to U_1\amalg U_2$
induce maps which for simplicity we again denote $s,t: \sR\to \sU$.
They are clearly surjective. We claim that moreover they are \'etale.
This is a local property and is proved in \cite[\spref{08KQ}]{SP}.
%Moreover, one sees easily that the diagram:
%\[
%\xymatrix{
%R\amalg R \ar[r] \ar[d] & U \amalg U \ar[d] \\
%\sR \ar[r] & \sU}
%\]
%is cartesian, where the horizontal maps are either $s$ or $t$.
Let $W=\sU/\sR$ be the quotient algebraic space. Checking that
$W$ is the pushout is formal, and obtaining the additional
properties is easy by taking an atlas.
\end{proof}

We obtain at least a necessary condition.

\begin{lemma} \label{lemm:at_least}
An epimorphism of algebraic spaces does not factor through
a locally closed subspace $Z\subsetneq S$.
\end{lemma}

\begin{proof}
Same proof as~\ref{lemma:epi_of_schemes} using
Lemma~\ref{lemm:gluing_along_closed_for_alg_spaces} instead
of \cite[Thm.~7.1]{Fe03}.
\end{proof}

We now present two simple examples of epimorphisms of algebraic
spaces. The first
one improves \cite[Prop.~7.2]{Ry10} where it is assumed that
$f$ is a submersion after every base change.

\begin{lemma} \label{lemm: e-submersion is epi}
Let $f:S'\to S$ be a morphism of algebraic spaces which is
schematically dominant, and submersive after every \'etale base
change on $S$. Then $f$ is an epimorphism of algebraic spaces, and
remains an epimorphism after every \'etale base change.
\end{lemma}

\begin{proof}
The assumptions are stable by \'etale base change, hence
it is enough to prove that $f$ is an epimorphism.
Let $X$ be an algebraic space and let $u,v:S\to X$ be morphisms
such that $uf=vf$. Let~$Z$ be the preimage of the diagonal
$\Delta:X\to X\times X$ by the map $(u,v):S\to X\times X$. Since
$\Delta$ is a representable monomorphism of spaces which is locally
of finite type, see
\cite[\spref{02X4}]{SP},
the map $g:Z\to S$ has the same properties. By the assumption on $u,v$ the
map $f$ factors through $Z$. This shows that $g$ is a submersive
monomorphism, hence a homeomorphism.
%Moreover, since $f$ is universally closed it is quasi-compact, see
%\cite[\spref{04XW}]{SP}.
%As $S'\to Z$ is surjective it follows that $Z\to S$ is quasi-compact.
By the assumption on $f$, this remains true after every \'etale base
change on $S$. Then \cite[Cor.~18.12.4]{EGA4.4}, whose proof uses only
\'etale base changes, shows that $g$ is finite. Thus $g$ is a closed
immersion which is schematically dominant, hence an
isomorphism. Hence $u=v$, and $f$ is an epimorphism of spaces.
\end{proof}

\begin{lemma} \label{lemma:epi_to_noetherian_local_scheme}
Let $S=\Spec(A)$ be a noetherian local scheme and let
$S_n=\Spec(A/m^{n+1})$ be the $n$-th thickening of the closed point.
Then $f:\coprod_{n\ge 0} S_n\to S$ is an epimorphism of algebraic spaces.
\end{lemma}

\begin{proof}
Since $f$ factors through the maximal-adic completion of $S$
which is fppf over $S$, it is enough to assume that $S$ is complete.
Let $u,v:S\to X$ be such that $uf=vf$, and $Z$ as in the proof
of~\ref{lemm: e-submersion is epi}. Since $S$ is Henselian
we can write $Z=Z_0\amalg Z_1$ where $Z_0$
is finite over $S$ and contains the unique closed point above the
closed point of $S$. By assumption $Z_0\to S$ is an isomorphism
over every $S_n$. Using Nakayama, we find that $Z_0\to S$ is
a closed immersion. Since $S$ is noetherian, this implies that
$Z_0\to S$ is an isomorphism.
\end{proof}

\begin{remarks}
The noetherian assumption is of course crucial, since otherwise we
may e.g.\ have $m=m^n$ for all $n\ge 1$.
\end{remarks}

\subsection{Effective epimorphisms}

\begin{definition}
We say that $f:S'\to S$ is an {\em effective epimorphism of
algebraic spaces} if the diagram $S'\times_S S'\rightrightarrows S'\to S$ is exact,
that is, if for all algebraic spaces $X$ we have an exact
diagram of sets:
\[
\Hom(S,X)\to\Hom(S',X)\rightrightarrows\Hom(S'\times_S S',X).
\]
\end{definition}

%\begin{exam*}
%Again let $k$ be a field of characteristic $p\ne 2$.
%Let $S$ be the plane cuspidal curve with equation $y^2=x^2(x+1)$
%and $i:S\to S$, $(x,y)\mapsto (x,-y)$ the involution.
%Let $S'=\AA^1\setminus \{-1\}=\Spec(k[t,(t+1)^{-1}])$ be the affine
%line minus a point. Let $f:S'\to S$, $t\mapsto (t^2-1,t(t^2-1))$
%be the morphism which is the inclusion of $S'$ as an open
%subset of the normalization of $S$. Let $X=S/R$ be the quotient of $S$
%by the \'etale equivalence relation obtained by identifying the points
%$(x,y)$ and $(x,-y)$ in $S$, for $y\ne 0$. Let $u:S\to X$ be the quotient
%map and $v=u\circ i$. Then $u\ne v$ but $uf=vf$, hence $f$ is not an
%epimorphism of algebraic spaces.
%\end{exam*}

Another way to say it is that $S$ is the categorical quotient of $S'$
by the groupoid $S'\times_SS'\rightrightarrows S'$.

\begin{exam}
An fpqc covering of algebraic spaces is an
effective epimorphism of algebraic spaces~\cite[\spref{04P2}]{SP}.
\end{exam}

If $f:X\to S$ is a morphism, we write $\cA_S(X)=f_*\cO_X$ or simply
$\cA(X)=f_*\cO_X$ if the base~$S$ is clear from context.
For instance $\cA(S)=\cO_S$. Also let us write $S''=S'\times_S S'$.

\begin{lemma} \label{lemma:effective_epi_and_functions}
Let $f:S'\to S$ be a quasi-compact and quasi-separated morphism of
algebraic spaces. Assume that $f$ is an effective epimorphism.
Then the sequence $\cA(S)\to\cA(S')\rightrightarrows\cA(S'')$ is exact.
\end{lemma}

\begin{proof}
Let us simplify the notations by setting
$\cA^*=\cA(S^*)$ for $*\in\{\varnothing,',''\}$. Let $\cI$ be the
kernel of $\cA\to\cA'$ and let $\cB$ be the kernel of the pair of
arrows $\cA'\rightrightarrows\cA''$. We must prove that $\cA\to\cB$ is an
isomorphism. Since $f$ is quasi-compact and quasi-separated, the
sheaves $\cA$, $\cA'$, $\cA''$ are quasi-coherent hence the sheaves
$\cI$, $\cB$ are also quasi-coherent. According to
Lemma~\ref{lemm:at_least} we have
$\cI=0$. Let us write $T=\Spec_S(\cB)$.
We have injective sheaf morphisms $\cO_S=\cA\to\cB\to\cA'$ and
corresponding scheme morphisms $g:S'\to T$,
$h:T\to S$ satisfying $f=hg$. Let $p_1,p_2:S''\to S'$ be the
projections. Since $gp_1=gp_2$ and $f$ is effective, there is
a morphism $e:S\to T$ such that $g=ef=ehg$. As the sheaf map
$g^\sharp:\cB\to\cA'$ is injective, this implies that
$e^\sharp:\cB\to\cA$ is a section of the map $h^\sharp:\cA\to \cB$
which therefore is an isomorphism.
\end{proof}

This lemma shows that under the qcqs assumption, it is
necessary for an effective epimorphism of algebraic spaces
to give rise to an {\em exact} sequence of
$\cO_S$-modules $\cA(S)\to\cA(S')\rightrightarrows\cA(S'')$.
For the converse, in the world of schemes things are quite
simple: a submersion with the above exact sequence property
is an effective epimorphism, see \cite[Exp.~VIII, Prop.~5.1]{SGA1}.

In the world of algebraic spaces things are a bit
more subtle, and our purpose in the rest of this subsection
is to strengthen slightly the submersion
property so as to salvage the result. We recall that to say
that $f:S'\to S$ is a morphism of effective descent for \'etale
algebraic spaces means that for any two \'etale $S$-algebraic
spaces $X,Y$ the diagram

\[\Hom_S(X,Y)\to\Hom_{S'}(X',Y')\rightrightarrows \Hom_{S''}(X'',Y'')
\]
is exact, and that for every \'etale $S'$-algebraic space $X'$, every
descent datum on $X'$ with respect to $S'\to S$ is effective.

\begin{lemma} \label{lemm:eff_descent_etale_local_on_target}
Let $f:S'\to S$ be a morphism of algebraic spaces.
The property for $f$ to be a morphism
of effective descent for \'etale
algebraic spaces is local on the source and target for
the \'etale topology. Explicitly,
\begin{enumerate}
\itemn{1} if $T\to S$ is \'etale surjective,
$T'=T\times_SS'$,
and $f_T:T'\to T$ is the pullback of $f$, then $f$ is a morphism of
effective descent for \'etale algebraic spaces
if and only if $f_T$ is so; and
\itemn{2} if $g:S''\to S'$ is \'etale surjective, then $f$ is a morphism of
effective descent for \'etale algebraic spaces
if and only if $fg$ is so.
\end{enumerate}
\end{lemma}

\begin{proof}
(1) In one direction, assume $f:S'\to S$ is a morphism of
effective descent for \'etale algebraic spaces,
and let $T\to S$ be an \'etale base change.
Let $T'=T\times_S S'$ and $T''=T'\times_TT'=T\times_S S''$.
We prove that $f_T:T'\to T$ descends morphisms. Let $X,Y$
be two \'etale $T$-algebraic spaces. We prove that the diagram
\begin{equation}\def\theequation{$\star$}
\Hom_T(X,Y)\to\Hom_{T'}(X',Y')\rightrightarrows \Hom_{T''}(X'',Y'')
\end{equation}
is exact. Note that $X\to T\to S$ is \'etale and similarly for the
other algebraic spaces. Since $f$ descends morphisms between \'etale
spaces, we obtain an exact diagram
\[
\Hom_S(X,Y)\to\Hom_{S'}(X',Y')\rightrightarrows \Hom_{S''}(X'',Y'').
\]
Injectivity of the first map of ($\star$) now follows from the injectivity
of the maps $\Hom_T(X,Y)\to \Hom_S(X,Y)$ and $\Hom_S(X,Y)\to \Hom_S(X',Y')$.
Let
$u':X'\to Y'$ be a $T'$-morphism such that its pullbacks under
the maps $T''\rightrightarrows T'$ coincide. The second exact sequence
provides an $S$-morphism $u:X\to Y$. Moreover
if $a:X\to T$, $b:Y\to T$ are the structure morphisms, we see
that $a$ and $bu$ become equal when pulled back to $S'$, hence
they are equal. This shows that $u$ is in fact a map of $T$-algebraic
spaces.
Finally we prove effective descent for objects. Let $X'\to T'$
be an \'etale algebraic space with a descent datum with respect
to $T'\to T$. Then $X'\to T'\to S'$ is \'etale and moreover
the descent datum can be viewed as a descent datum with respect
to $S'\to S$. By the assumption on~$f$ there exists an \'etale morphism
$X\to S$ whose pullback under $S'\to S$ is $X'$. Moreover the map
$X'\to T'$ descends to an $S$-map $X\to T$ and the construction of
$X$ is finished.

The other direction is a special case of~\cite[Thm.~10.8]{Gi64} but for
the convenience of the reader we give the argument here. Let $T\to S$ be
\'etale surjective and assume
that the base change $f_T:T'\to T$ is of effective descent for
\'etale algebraic spaces. We prove descent of morphisms
for $f$. Let $X,Y$ be \'etale spaces over
$S$, let $X',Y'$ be the pullbacks to $S'$, and let
$u':X'\to Y'$ be an $S'$-morphism whose pullbacks via the
two maps $S'\times_SS'\rightrightarrows S'$ coincide. Then the map
$u'_T$ obtained by the base change $T'\to S'$ has
coinciding pullbacks via the two maps $T'\times_TT'\rightrightarrows T'$.
Since $f_T$ descends morphisms, $u'_T$ descends to a $T$-map
$u_T:X_T\to Y_T$. Let us introduce some notation:
\[
\xymatrix{
T'\times_{S'}T' \ar@<.5ex>[r]^-{q_1} \ar@<-.5ex>[r]_-{q_2}
\ar[d]^{f_{T\times_S T}}
& T' \ar[r] \ar[d]^{f_T} & S' \ar[d]^f \\
T\times_ST \ar@<.5ex>[r]^-{p_1} \ar@<-.5ex>[r]_-{p_2}
& T \ar[r] & S.
}
\]
From the first part, we know that $f_{T\times_ST}$ is a morphism
of (effective) descent.
From the equality $q_1^*u'_T=q_2^*u'_T$ we thus deduce that
$p_1^*u_T=p_2^*u_T$. By descent along the \'etale map $T\to S$,
we obtain a unique $S$-map $u:X\to Y$ that descends $u'$.
% (To prove that $u$ does descend $u'$ one really also uses that
% $T'\to S'$ is a morphism of descent.)
Now we prove effective descent for objects.
Let $X'\to S'$ be an \'etale morphism equipped with a descent
datum for $S'/S$. The pullback $X'_T\to T'$ has a descent datum
for $T'/T$. By assumption it descends to $X_T\to T$. The
canonical isomorphism $q_1^*X'_T\to q_2^*X'_T$ descends to an isomorphism
$\psi\colon p_1^*X_T\to p_2^*X_T$ since $f_{T\times_S T}$ is a morphism of
descent.
Using that $f_{T\times_S T\times_S T}$ is a morphism of descent,
one checks that $\psi$ is a descent datum on
$X_T$ for the \'etale covering $T\to S$ and by effective descent,
it descends to a unique $X\to S$ as desired.

(2) This is a special case of~\cite[Props.~10.10 and 10.11]{Gi64}.
\end{proof}

The next theorem is our main result on effective epimorphisms
of algebraic spaces. We write {\em qcqs} for {\em quasi-compact and
quasi-separated}. In the world of schemes, a qcqs submersion
such that $\cA(S)\to\cA(S')\rightrightarrows\cA(S'')$ is exact is an effective
epimorphism. In the world of algebraic spaces, we reinforce these
conditions slightly in order to suitably allow \'etale
localization and descent.

\begin{theorem} \label{theo:eff_epi_of_spaces}
Let $f:S'\to S$ be a morphism of algebraic spaces. Assume that:
\begin{trivlist}
\itemn{1} $f$ is a qcqs
submersion and remains so after every \'etale base change,
\itemn{2} the diagram of $\cO_S$-modules
$\cA(S)\to\cA(S')\rightrightarrows\cA(S'')$ is exact,
\itemn{3} $f$ is a morphism of effective descent for \'etale
algebraic spaces.
\end{trivlist}
Then $f$ is an effective epimorphism of algebraic spaces and
remains so after any \'etale base change.
\end{theorem}

\begin{proof}
By the lemma, all three assumptions are stable by \'etale
base change on $S$. Therefore it is sufficient to prove
that $f$ is an effective epimorphism of algebraic spaces,
i.e., for all algebraic spaces $X$, the diagram
$X(S)\to X(S')\rightrightarrows X(S'')$ is exact. 
Note that after Lemma~\ref{lemm: e-submersion is epi}
we know that $f$ is an epimorphism after every \'etale base
change, which settles injectivity on the left. It remains
to prove that if $\alpha':S'\to X$
satisfies $\alpha'\pr_1=\alpha'\pr_2$ then there exists
$\alpha:S\to X$ such that $\alpha'=\alpha f$.

We prove that the question is Zariski-local on $X$. Let
$(X_i)$ be a covering of $X$ by open subspaces and let
$S'_i=(\alpha')^{-1}(X_i)$. Then $S'_i$ is saturated, that
is $S'_i=f^{-1}(f(S'_i))$. Since~$f$ is a submersion
by~(1),
then $S'_i$ descends to an open subspace $S_i\subset S$.
If for each $i$ there exists $\alpha_i:S_i\to X_i\subset X$
such that $\alpha'{}_{|S'_i}=\alpha_i f_{|S'_i}$ then by
uniqueness the morphisms
$\alpha_i$ glue to give a solution $\alpha:S\to X$.

We prove that the question has a positive answer when $X$
is a scheme. Indeed, we can cover~$X$ by open affine
subschemes and then by the preceding step we can reduce
to the case where $X=\Spec(A)$ is affine.
Since $\Hom(T,\Spec(A))=\Hom(A,\Gamma(T,\cO_T))$ for all
algebraic spaces $T$ (see
\cite[\spref{05Z0}]{SP}),
the question reduces to a construction of ring homomorphisms and
then the conclusion comes from assumption~(2).

Now let $X$ be an arbitrary algebraic space.
%Using a covering by quasi-compact open subspaces we can
%assume that $X$ is quasi-compact.
Let $\pi:Y\to X$ be an \'etale surjective morphism where $Y$
is a scheme. Let $U'=Y\times_X S'$ which is \'etale surjective
over $S'$, and $U''=Y\times_X S''$. The assumption
$\alpha'\pr_1=\alpha'\pr_2$ implies
that $U'$ carries a descent datum. By assumption (3) it
descends to an \'etale algebraic space $U\to S$. Also let
$\beta':U'\to Y$ be the pullback of
$\alpha'$. Let $R=U\times_S U$ and $R'=U'\times_{S'} U'$.
\[
\xymatrix{
& R' \ar@<.5ex>[d]^-{t'} \ar@<-.5ex>[d]_-{s'} \ar[r]
& R \ar@<.5ex>[d]^-t \ar@<-.5ex>[d]_-s & \\
U'' \ar@<.5ex>[r] \ar@<-.5ex>[r] \ar[d]^c
& U' \ar@{-->}[r] \ar[d]^d \ar@/_1pc/[rr]_(.7){\beta'}
& U \ar@{-->}[r]^{\beta} \ar@{-->}[d] & Y \ar[d]^{\pi} \\
S'' \ar@<.5ex>[r] \ar@<-.5ex>[r] & S' \ar[r] \ar@/_1pc/[rr]_(.7){\alpha'} & S & X}
\]
We know $\beta'\pr_1=\beta'\pr_2:U''\to Y$.
Since $U'\to U$ satisfies again all the assumptions (1)--(3)
and the statement holds when the test space $Y$ is a scheme,
we obtain a morphism $\beta:U\to Y$. We claim that
$\pi\beta:U\to X$ is $R$-invariant. Since $R'\to R$ is an \'etale
pullback of $f:S'\to S$, it is an epimorphism.
Hence it is enough to prove that the compositions
$R'\to R \rightrightarrows U \to X$ are equal. This follows because
they equal to $\alpha'ds'=\alpha'dt'$.
Thus $\pi\beta$ induces a morphism
$\alpha:S\to X$ and we are done.
\end{proof}

%\begin{remark}
%According to David Rydh, universally submersive and quasi-compact
%morphisms of spaces are of effective descent for \'etale qcqs
%morphisms of spaces (private communication). Using
%Lemma~\ref{lemm: e-submersion is epi} to check
%condition (3) in Theorem~\ref{theo:eff_epi_of_spaces}, we obtain:
%if $f$ is a quasi-compact and quasi-separated universal submersion
%such that $\cA(S)\to\cA(S')\rightrightarrows\cA(S'')$ is exact, then $f$ is an
%effective epimorphism of algebraic spaces and remains
%so after flat base change.
%\end{remark}

Collecting some results on morphisms of effective descent for \'etale
maps in the literature, we find the following special cases.

\begin{corollary} \label{coro:cases_of_eff_descent}
Let $f:S'\to S$ be a surjective morphism of algebraic spaces
which is either~:
\begin{trivlist}
\itemm{i} integral,
\itemm{ii} proper,
\itemm{iii} universally open and locally of finite presentation,
\itemm{iv} universally submersive and of finite presentation with
$S$ locally noetherian.
\end{trivlist}
Then if the sequence of modules $\cA(S)\to\cA(S')\rightrightarrows\cA(S'')$ is exact,
the map $f$ is an effective epimorphism of algebraic spaces and remains
so after any flat base change.
\end{corollary}

\begin{proof}
In each case the assumptions are stable under base change,
except possibly in case~(iv). To deal with this, we use the notion
of a {\em subtrusive} morphism from \cite{Ry10} and
we replace (iv) with the more general (iv)'~: universally subtrusive and
of finite presentation. That this is indeed more general than (iv)
follows from \cite[Cor.~2.10]{Ry10}, with the advantage that (iv)'
is stable under base change. It follows that it is enough to prove
that $f$ is an effective epimorphism of algebraic spaces. For this
we apply Theorem~\ref{theo:eff_epi_of_spaces}. In each case
conditions~(1) and (2) hold and it remains to see that $f$ is of
effective descent for \'etale algebraic spaces. Since by
\ref{lemm:eff_descent_etale_local_on_target} this property
is \'etale-local on source and target, by taking \'etale atlases of $S$
and $S'$
one reduces to the case where $f$ is a map of schemes. Then the claim
is \cite[Exp.~VIII, Thm.~9.4]{SGA4.2} in cases (i)--(ii)
and \cite[Thm.~A.2]{Ry13} in cases (iii)--(iv)'.
\end{proof}

\begin{remark}
Assume that $f$ satisfies one of the
conditions (i)--(iv). Then the property ``$f$ is an effective
epimorphism'' is fpqc-local on $S$ because exactness of a sequence
of quasi-coherent modules is an fpqc-local condition.
\end{remark}

For ease of future reference, we single out the following
particular case of~\ref{coro:cases_of_eff_descent}. Recall
that an epimorphism is {\em uniform} if it remains an
epimorphism after all flat base changes.

\begin{corollary} \label{coro:uniform_effective_epi}
Let $f:S'\to S$ be an integral morphism of algebraic spaces such
that the sequence $\cA(S)\to\cA(S')\rightrightarrows \cA(S'')$ is exact.
Then $f$ is a uniform effective epimorphism of algebraic spaces.
\hfill $\square$
\end{corollary}

\begin{examples} \label{examples:effective_ring_extensions}
Here are some sufficient conditions for a morphism
$f:\Spec(A')\to \Spec(A)$ defined by a finite ring extension
$A\subset A'$ to be an effective epimorphism.
\begin{trivlist}
\itemn{1} $f$ is faithfully flat (faithfully flat descent).
\itemn{2} $f$ is the quotient of a flat groupoid (by the quotient
property).
\itemn{3} $f$ is unramified with fiber-degree at most 2.
Indeed, by the structure of unramified morphisms, \'etale-locally
on the target the morphism $f$ has the form
$\Spec(A/I)\amalg\Spec(A/J)\to \Spec(A)$. Hence we
may assume that $A'=A/I\times A/J$ with $I\cap J=0$, so that
$A'\otimes_AA'=(A/I)\times (A/I+J)\times (A/I+J)\times (A/J)$.
To say that $a'=(a_1,a_2)\in A'$ has equal images in
$A'\otimes_AA'$ means that $a_1\equiv a_2\mod I+J$, hence
$a_1+i=a_2+j$ for some $i\in I$, $j\in J$. Thus $a'\in A$.
\itemn{4} Levelt \cite{Le65} contains some more examples.
For instance if $A\subset A'$ is a local inclusion of local rings
with trivial residue field extension
and no intermediate subring then $f$ is effective
\cite[Chap.~IV, Lem.~4]{Le65}. If for some maximal ideal
$m\subset A$ we have $A'/A\simeq A/m$ as $A$-modules, then $f$
is effective \cite[Chap.~IV, Lem.~7]{Le65}.
\itemn{5} $f$ is weakly normal, e.g., $A$ and $A'$ are integral domains, $f$ is
generically \'etale and $A$ is weakly normal \cite[Lem.~B.5]{Ry10}.
\end{trivlist}
\end{examples}

Here is a non-example showing that $d=2$ is required in (3) above.
\begin{exam}
Let $A=k[x,y]/\bigl(xy(y-x)\bigr)$ and $A'=A/(x)\times A/(y)\times
A/(y-x)$. Then $f:\Spec(A')\to \Spec(A)$ is finite and unramified of
fiber-degree at most $3$ but not an effective epimorphism.  Indeed,
$A'\otimes_A A'=A'\times k^6$ is reduced so the equalizer of the two maps
$A'\to A'\otimes_A A'$ is the weak subintegral closure~\cite[Lem.~B.5]{Ry10}
which is isomorphic to $B=k[u,v,w]/(u,v)(u,w)(v,w)$. Explicitly, we have
injective maps $A\to B$ and $B\to A'$ where $x\mapsto u+v$, $y\mapsto u+w$ and
$u\mapsto (0,0,x)$, $v\mapsto (0,x,0)$, $w\mapsto (y,0,0)$.
\end{exam}

\subsection{The canonical factorization}

The main result of this section gives a canonical
factorization of a finite epimorphism as a composition of
finitely many finite effective epimorphisms. It is first stated
in \cite[A.2.b]{Gr59} and then used to study
the functor of subgroups of multiplicative type of a group scheme
\cite[Exp.~XV, just before Lem.~3.7]{SGA3.2} and the relative
representability of the Picard functor \cite[Exp.~XII,
Lem.~2.6]{SGA6}. A proof appears in the latter reference.
With an eye towards
the study of groupoids of higher complexity, we provide
additional properties of the canonical factorization~:
uniqueness, compatibility with flat base change, minimality
of its length. For the convenience of the reader, we provide
complete proofs.

\begin{definitions}
Let $f:T\to S$ be an epimorphism of algebraic spaces.
\begin{trivlist}
\itemn{1}
An {\em $f$-sequence} is a sequence
$T=T_0\to T_1\to T_2\to\dots$ of epimorphisms of $S$-spaces
such that for each $i\ge 0$, if $T_i\to T_{i+1}$ is an isomorphism
then $T_{i+1}\to T_{i+2}$ is an isomorphism.
\itemn{2} The {\em length} of an $f$-sequence as above is the
smallest $n\in\NN\cup\{\infty\}$ such that $T_n\to T_{n+1}$
is an isomorphism, i.e., the number of non-isomorphic arrows
of the sequence. If an $f$-sequence has finite length~$n$ and
$T_n\to S$ is an isomorphism, we say that it is
{\em finite and separated} or that it is a {\em factorization}.
%The factorization is called {\em separated} if
%$\cap_{i\ge 0} \cA_i=0$ where we write $T_i=\Spec_S(\cA_i)$.
\itemn{3} Assume that $f$ is affine. The {\em canonical sequence}
of $f$ is the $f$-sequence $T=T_0\to T_1\to T_2\to\dots$ given
by $T_i=\Spec_S(\cA_i)$ where $\cA_0:=f_*\cO_T$ and
$\cA_{i+1}:=\ker(\cA_i\rightrightarrows \cA_i\otimes_{\cO_S} \cA_i)$
for all $i\ge 0$.
\end{trivlist}
\end{definitions}

\begin{prop} \label{prop:canonical_factorization}
Let $f:T\to S$ be an integral epimorphism of algebraic spaces.
\begin{trivlist}
\itemn{1} The canonical sequence
$T=T_0\to T_1\to T_2\to\dots$
is characterized by the properties~:
\begin{trivlist}
\itemm{i} for each $i$, the morphism $T_i\to S$ is integral and
the morphism $T_i\to T_{i+1}$ is an
integral effective epimorphism;
\itemm{ii} for each $i$ the canonical morphism
$T_i\times_{T_{i+1}} T_i\to T_i\times_S T_i$ is an isomorphism.
\end{trivlist}
\itemn{2} The formation of the canonical sequence is compatible
with flat base change and local for the flat topology on~$S$.
More precisely, let $S'\to S$ be a faithfully flat morphism
of schemes.
Let $\sT=(T_0\to T_1\to T_2\to\dots)$ be a sequence of morphisms
of $S$-schemes and let $\sT'=(T'_0\to T'_1\to T'_2\to\dots)$ be
the sequence obtained by the base change $S'\to S$. Then $\sT$
is the canonical sequence of $T\to S$ if and only if $\sT'$
is the canonical sequence of $T'\to S'$.
\itemn{3} The canonical sequence has length~$0$ if and only if $f$ is
an isomorphism, and length at most $1$ if and only
if $f$ is an effective epimorphism. 
\itemn{4} The canonical sequence is terminal among $f$-sequences, affine
over $S$,
whose factors are effective epimorphisms, i.e., for each such sequence
$T=T'_0\to T'_1\to T'_2\to\dots$ there are maps $T'_i\to T_i$
making a commutative diagram:
\[
\xymatrix{
T'_0 \ar[r] \ar[d]_{\id_T} & T'_1 \ar[r] \ar[d]
& T'_2 \ar[r] \ar[d] & \dots \\
T_0 \ar[r] & T_1 \ar[r] & T_2 \ar[r] & \dots}
\]
\itemn{5} When $S$ is noetherian and $f$ is finite, the morphisms
$T_i\to T_{i+1}$ are finite and the canonical sequence
is finite and separated, i.e., a factorization. It has minimal
length among all finite separated $f$-sequences whose
factors are effective epimorphisms.
\end{trivlist}
\end{prop}

All claims except~(5) are actually quite formal.

\begin{proof}
(1) Write $\cA:=\cO_S$ and $\cA_0:=f_*\cO_T$. Since $\cA_i\subset \cA_0$,
the morphisms $T_i=\Spec_S \cA_i\to S$ and $T_i\to T_{i+1}$ are integral.
The surjective
morphism $\cA_i\otimes_{\cA} \cA_i\to \cA_i\otimes_{\cA_{i+1}} \cA_i$
has kernel generated by the local sections
$a\otimes 1-1\otimes a$ for local sections $a\in \cA_{i+1}$.
By the definition of $\cA_{i+1}$, it follows that this map is an
isomorphism hence~(ii) is satisfied. Therefore we have an exact
diagram $\cA_{i+1}\to \cA_i
%\:\substack{\too\\[-1em] \too}\:
\rightrightarrows \cA_i\otimes_{\cA_{i+1}} \cA_i$. By
Corollary~\ref{coro:uniform_effective_epi}, this means
that $T_i\to T_{i+1}$ is an effective
epimorphism, hence~(i) is satisfied. Conversely, if a factorization
$T=T'_0\to T'_1\to T'_2\to\dots$ satisfies (ii) then
$\cA'_i\otimes_{\cA} \cA'_i\to \cA'_i\otimes_{\cA'_{i+1}} \cA'_i$
is an isomorphism, and if moreover (i) is satisfied then
$\cA'_{i+1}=\ker(\cA_i\rightrightarrows \cA_i\otimes_{\cO_S} \cA_i)$.
Thus we see that the given sequence is the canonical one.

\smallskip

\noindent (2) This follows because the formation of kernels
of morphisms of quasi-coherent sheaves commutes with flat
base change and is local for the flat topology on the base.

\smallskip

\noindent (3) This follows from the definitions.

\smallskip

\noindent (4) By induction, assume that there is a diagram
of length $i$:
\[
\xymatrix{
T'_0 \ar[r] \ar[d]_{\id_T} & T'_1 \ar[r] \ar[d]
& \dots \ar[r] & T'_{i-1} \ar[r] \ar[d] & T'_i \ar[d] \\
T_0 \ar[r] & T_1 \ar[r] & \dots \ar[r] & T_{i-1} \ar[r] & T_i.}
\]
Then because $T'_i\to T'_{i+1}$ is effective, we have a containment
$\cA_{i+1}=\ker(\cA_i\rightrightarrows \cA_i\otimes_{\cA} \cA_i)\subset
\ker(\cA'_i\rightrightarrows \cA'_i\otimes_{\cA'_{i+1}} \cA'_i)=\cA'_{i+1}$.
This gives a map $T'_{i+1}\to T_{i+1}$ and a diagram
of length $i+1$.

\smallskip

\noindent (5) First, assume that the canonical sequence has
finite length, so there exists $n\ge 0$ such that $\cA_{n+1}=\cA_n$.
Then we have an isomorphism
$\cA_n\otimes_{\cA} \cA_n\to
\cA_n\otimes_{\cA_{n+1}} \cA_n\simeq\cA_n$.
This means that $T_n\to S$ is a monomorphism. Being dominant
and finite, it must be an isomorphism hence the sequence
is separated. Now we prove that the canonical sequence
has finite length. Since~$S$ is noetherian, this property
is \'etale-local on $S$. Moreover the formation of~$\cA_i$
commutes with restriction to an open subscheme and with passage
to the stalks on \'etale local rings.
If for some point $s\in S$ the sequence of stalks
$(\cA_{i,s})_{i\ge 0}$ is stationary, then the isomorphism
$\cA_s\to \cA_{n,s}$ extends in a neighborhood of $s$. Thus we
may assume that~$S$ is local with closed point $s$. In particular,
we may assume that $S$ (local or not) has finite dimension~$d$.
We now argue by induction on $d$. If $d=0$, the rings $\cA_i$ have
finite length and the sequence $\cA_i$ is stationary. If $d>0$,
the open $U=S\setminus \{s\}$ has dimension $<d$ so by induction
the sequence $\cA_{i}$ is stationary after restriction to~$U$.
By the same argument as before, we then know that for all big enough
$i$ the morphism $\cA\to\cA_i$ is an isomorphism away from $s$.
It follows that the quotient $\cO_S$-module $\cA_i/\cA$ has finite
length. Thus $\cA_i/\cA$ is stationary, and hence also $\cA_i$.

Now consider a finite separated sequence
$T=T'_0\to T'_1\to \dots\to T'_m=S$ of length $m$ whose factors
are effective epimorphisms. We have a diagram:
\[
\xymatrix{
T'_0 \ar[r] \ar[d]_{\id_T} & T'_1 \ar[r] \ar[d]
& \dots \ar[r] & T'_{m-1} \ar[r] \ar[d]
& T'_m \ar[d] \ar@{=}[r] & S \ar@{=}[d] \\
T_0 \ar[r] & T_1 \ar[r] & \dots \ar[r] & T_{m-1} \ar[r] & T_m \ar[r] & S.}
\]
We obtain $\cO_S\subset \cA_m\subset \cA'_m=\cO_S$. Thus
$T_m\to S$ is an isomorphism, so the canonical sequence
has length at most $m$.
\end{proof}

\begin{exam}
Let $k$ be a field and $S$ the affine cuspidal $k$-curve with equation
$y^3=x^4$. We shall see that the canonical sequence of the
normalization map $f:T\to S$ has length $n=2$, as follows~:

\begin{center}
\begin{tikzpicture}[scale=1.3,text height=1.5ex, text depth=0.5ex]
\node (A) at (0,0) {$T$};
\node (B) at (2,0) {$T_1$};
\node (C) at (4,0) {$S$};
\node (D) at (0,-.8) {$\AA^1_k$};
\node (D) at (2,-1.1)
{{\footnotesize \begin{tabular}{c}
 spatial \\ singularity \\
$\begin{array}{l} y^2=xz \\ z^2=x^2y \\ yz=x^3 \end{array}$
\end{tabular}}};
\node (D) at (4,-1) {\footnotesize \begin{tabular}{c}
planar \\ cuspidal \\ singularity \\
$y^3=x^4$ \\
\end{tabular}};
\draw[-myto,font=\scriptsize]
(A) edge (B)
(B) edge (C);
%(B) edge (D)
%(C) edge node[pos=.45,above] {$\alpha$} (D);
\end{tikzpicture}
\end{center}

\medskip\bigskip

\noindent We have $S=\Spec(A)$ and $T=\Spec(B)$ with
$A=k[x,y]/(y^3-x^4)$ and $B=k[t]$, the morphism $A\to B$ being
given by $x=t^3$ and $y=t^4$. In other words
$A\simeq k[t^3,t^4]\into k[t]$. We can write~:
\[
B\otimes_A B=\frac{k[t_1,t_2]}{(t_1^3-t_2^3,t_1^4-t_2^4)}
\]
and the two arrows $B\rightrightarrows B\otimes_A B$ map $t$ to $t_1$ and $t_2$
respectively. The ring $B_1=\ker(B\rightrightarrows B\otimes_A B)$ contains $A$
as well as the element $t^5$, since
$t_1^5=t_1t_2^4=t_1t_2t_1^3=t_1^4t_2=t_2^5$. Therefore $B$ contains
$k[t^3,t^4,t^5]$. If we notice that the annihilator of
$t_1-t_2$ in $B\otimes_A B$ is generated by
$t_1^2+t_1t_2+t_2^2$ and $(t_1+t_2)(t_1^2+t_2^2)$, we see that $B_1$
does not contain elements of the form $at+bt^2$. This proves that
$B_1=k[t^3,t^4,t^5]$. Letting $z=t^5$ we get the presentation~:
\[
B_1=\frac{k[x,y,z]}{(y^2-xz,z^2-x^2y,yz-x^3)}.
\]
In particular $B_1$ is a free $k[x]$-module with basis
$\{1,y,z\}$. We now prove that $A=\ker(B_1\rightrightarrows B_1\otimes_A B_1)$.
We write~:
\[
B_1\otimes_A B_1=k[x,y,z_1,z_2]/I
\]
with $I=(y^2-xz_1,z_1^2-x^2y,yz_1-x^3,x(z_1-z_2),y(z_1-z_2),z_1^2-z_2^2)$.
The two arrows $B_1\rightrightarrows B_1\otimes_A B_1$ map $z$ to $z_1$ and $z_2$
respectively. Let $P=a(x)+b(x)y+c(x)z$ be an element of $B_1$ such
that $P(x,y,z_1)=P(x,y,z_2)$, i.e., $c(x)z_1=c(x)z_2$. In view of the
structure of the annihilator of $z_1-z_2$ inside $B_1\otimes_A B_1$,
this implies that $x$ divides $c(x)$, hence $P\in k[x,y,xz]=k[x,y]=A$,
as announced.
\end{exam}

\section{The category of groupoids} \label{section:groupoids}

In this section we briefly recall some definitions and notations
on groupoids (\S~\ref{subsection:definitions}) and we define the
complexity of a flat groupoid with finite stabilizer whose
$j_Y:R\to X\times_Y X$ map is schematically dominant (\S~\ref{subsection:complexity}).

\subsection{The vocabulary of groupoids}
\label{subsection:definitions}

Good references for this material are Keel--Mori \cite{KM97} and
Rydh \cite{Ry13}. We fix a base algebraic space~$S$, and products
are fibered over $S$. We use the vocabulary of the functor of
points: a {\em $T$-point} of an algebraic space~$X$ over $S$
is a morphism $x:T\to X$ with values in some $S$-scheme $T$.
We often write $x\in X(T)$.

\begin{thematic-item} {\bf Groupoids.}
We work with {\em groupoids in $S$-algebraic spaces}, also called
{\em groupoid spaces} or simply {\em groupoids}. A groupoid is given
by five morphisms of algebraic $S$-spaces $s,t:R\to X$, $c:R\times_{s,X,t} R\to X$,
$e:X\to R$, $i:R\to R$ subject to the conditions that $X(T)$ is the
set of objects and $R(T)$ is the set of arrows of a small category,
functorially in $T$. The maps $s,t,c,e,i$ are called {\em source},
{\em target}, {\em composition}, {\em unit} (or {\em identity}),
{\em inversion}. The points of
$R\times_{s,X,t} R$ are called {\em pairs of composable arrows}. Usually
we denote a groupoid simply by $s,t:R\rightrightarrows X$ and we call $j$ the map
$j=(t,s):R\to X\times X$. Typically a $T$-point of $X$ will be denoted~$x$
while a $T$-point of $R$ will be denoted with a Greek letter like $\alpha$.
We sometimes write~$1_x$ or simply $1$ instead of $e(x)$.
We occasionally write $\alpha:x\to y$ if $x=s(\alpha)$ and $y=t(\alpha)$.
With our choices of $c$ and $j$, note that it is more natural to picture
$T$-points of $R$ as arrows $y\xleftarrow{\alpha} x$ going from right to left.
\end{thematic-item}

\begin{thematic-item} {\bf Actions.}
For instance, an $S$-group space $G$ acting on an algebraic space $X$ gives
rise to a groupoid $s,t:G\times X\rightrightarrows X$ where~$s$ is the second
projection and $t$ is the action. In the general setting one may
shape one's intuition by thinking of a groupoid as a space $R$
acting on a space $X$. If $\alpha:x\to y$ is an arrow, there is a
corresponding action-like notation $\alpha(x):=t(\alpha)=y$. In
these terms, the action is trivial if and only if $s=t$ and the
maps $c,e,i$ make $R\to X$ into an $X$-group space.
\end{thematic-item}

\begin{thematic-item} {\bf Stabilizers.}
If $R\rightrightarrows X$ is a groupoid, then its
{\em stabilizer} is the $X$-group space $\Stab_R=j^{-1}(\Delta_X)$
where $\Delta_X\subset X\times X$ is the diagonal. This is the
largest subgroupoid of $R$ which is a group space, or also, the
largest subgroupoid acting trivially.
\end{thematic-item}

\begin{thematic-item} {\bf Subgroupoids.} \label{item:subgroupoids}
A {\em subgroupoid} is a sub-algebraic space $P\subset R$ that is stable
under composition and inversion, and contains the unit section $e(X)$.
(Topologists call this a {\em wide subgroupoid} because they also allow
subgroupoids $P\rightrightarrows Y$ whose base is an arbitrary possibly empty subspace
$Y\subset X$. By sub-algebraic space, we here mean a subfunctor that is an
algebraic space, that is, a monomorphism $P\to R$ of algebraic spaces.)
A subgroupoid is called {\em normal} if for any
$\alpha\in P(T)$ and $\varphi\in R(T)$ we have $\varphi\alpha\varphi^{-1}\in P(T)$
whenever composability holds. In detail, if $\varphi:x\to y$, then
composability means that $\alpha\in\Stab_{P,x}(T)$ and then we have
$\varphi\alpha\varphi^{-1}\in\Stab_{P,y}(T)$. In particular the condition
that~$P$ be normal in $R$ depends only on the stabilizer $\Stab_P$. Any
subgroupoid containing $\Stab_R$ is normal; in particular if $\Stab_R$
is trivial then all subgroupoids are normal.
\end{thematic-item}

\begin{thematic-item} {\bf Morphisms, kernels.}
A {\em morphism of groupoids} from $R\rightrightarrows X$ to $R'\rightrightarrows X'$
is a morphism of $S$-spaces $f:R\to R'$ such that
$f(\alpha\beta)=f(\alpha)f(\beta)$ for all composable arrows
$\alpha,\beta\in R(T)$. We also use the notation $f:(R,X)\to (R',X')$.
Such a morphism $f$ has various automatic
compatibilities with the maps $s,t,e,i$. For instance, $f$ maps
identities to identities. Moreover there is an induced morphism on objects
$s'\circ f\circ e=t'\circ f\circ e:X\to X'$ which we also write $f$
for simplicity. Thus, notationally for an arrow $\alpha:x\to y$ in $R$
we obtain an arrow $f(\alpha):f(x)\to f(y)$ in $R'$.
The {\em kernel} of a morphism $f:R\to R'$ is the preimage of the unit
section $e':X'\to R'$. It is a normal subgroupoid of $R$.
\end{thematic-item}

\begin{thematic-item} {\bf Invariant morphisms.}
Let $R\rightrightarrows X$ be a groupoid and let $P$ be a subgroupoid. Then
$P$ acts on $R$ in various natural ways. The action by precomposition
is a groupoid $R\times_{(s,t)}P\rightrightarrows R$, and the action by
postcomposition is a groupoid $P\times_{(s,t)}R\rightrightarrows R$.
The stabilizers of both actions are trivial. The simultaneous action,
to be called {\em by pre-post-composition}, is a groupoid
$P\times_{(s,t)}R\times_{(s,t)}P\rightrightarrows R$. We have an isomorphism
$\Stab_{P\times_{(s,t)}R\times_{(s,t)}P} \isomto
\Stab_P\times_{(s,t)}R$
given by $(\varphi,\alpha,\psi)\longmapsto (\varphi,\alpha)$.
This implies that the morphism of groupoids
$f:P\times_{(s,t)}R\times_{(s,t)}P \too R$, $f(\varphi,\alpha,\psi)=\varphi$
whose underlying morphism on objects is $f=t:R\to X$ is fixed point reflecting,
in the sense of \cite[2.2]{KM97}. Now let us consider moreover a morphism
of groupoids $f:R\to R'$. Then the following four assertions are
rewordings of one and
the same property~: (i) $P\subset \ker(f)$, (ii) $f$ is invariant by
the left $P$-action on $R$, (iii) $f$ is invariant by the right $P$-action
on $R$, (iv) $f$ is invariant by the pre-post-composition $P$-action
on $R$. If this property holds, we say that~$f$ is {\em $P$-invariant}.
\end{thematic-item}

\begin{thematic-item} {\bf Quotients.} \label{def:quotient}
Let $R\rightrightarrows X$ be a groupoid and $P\subset R$ a subgroupoid.
A {\em categorical quotient} of $R$ by~$P$ is a morphism of groupoids
$\pi:R\to Q$ which is $P$-invariant and is universal among invariant
morphisms $R\to R'$. A categorical quotient is
called {\em regular} if the canonical map $P\to \ker(\pi)$ is an isomorphism.
\end{thematic-item}

In Definition~\ref{def:quotient} we simplify the discussion by
restricting to categorical quotients; other notions of quotients
are recalled in~\ref{subsection:complexity} below. To shed light
on the definition, note that by the universal property there is
a morphism $P\to \ker(\pi)$ but contrary to what happens in the
category of groups, it is not at
all clear if this is an isomorphism (and we do not think it is the
case in general). We will not pursue this question in this article.

\subsection{The complexity} \label{subsection:complexity}

Whereas we introduced basic notions internal to the category of
{\em groupoids}, in order to define the complexity we come back to
the categories of schemes and algebraic spaces. Recall that if
$s,t:R\rightrightarrows X$ is a groupoid space, then a morphism $f:X\to X'$ is
called {\em $R$-invariant} if $fs=ft$. We will not repeat here the
various definitions related to quotients because they receive a clear
presentation in \cite[\S~1]{KM97} and \cite[\S~2]{Ry13}. We content
ourselves with saying that a morphism $X\to Y$ is a {\em categorical
quotient} if it is initial among $R$-invariant morphisms $X\to X'$,
a {\em geometric quotient} if it is a submersion and $\cO_Y$ is
identified with the sheaf of $R$-invariant sections of $\cO_X$, and
a quotient of one of these types is {\em uniform} it its formation
commutes with flat base change. We recall the statement of the
fundamental Keel--Mori theorem from \cite{KM97}, \cite{Ry13} as well
as the case with trivial stabilizer from \cite{Ar74}.

\begin{theorem} \label{theo:quotient_by_ff_groupoid}
Let $S$ be an algebraic space and let $R\rightrightarrows X$ be a flat,
locally finitely presented $S$-groupoid space with finite
stabilizer.
\begin{trivlist}
\itemn{1} There is a uniform geometric and categorical quotient
$X\to X/R=Y$ such that the map $j_Y:R\to X\times_Y X$ is finite
and surjective. Moreover $X\to Y$ is universally open.
\itemn{2} The space $Y\to S$ is separated (resp.\ quasi-separated)
if and only if $j_S:R\to X\times_SX$ is finite (resp.\ quasi-compact).
It is locally of finite type if $S$ is locally
noetherian and $X\to S$ is locally of finite type.
\itemn{3} If the stabilizer is trivial, then $Y$ is the fppf
quotient sheaf of $X$ by $R$, $X\to Y$ is flat locally finitely
presented, $j_Y$ is an isomorphism, and the formation
of $Y$ commutes with arbitrary base changes $Y'\to Y$.
\end{trivlist}
\end{theorem}

When $R\rightrightarrows X$ is finite and locally free, it is known moreover
that $X\to Y$ is integral.

\begin{remarks} \label{examples of quotient}
(1) The map $j_Y:R\to X\times_Y X$ need not be schematically
dominant, in particular it need not be an epimorphism. Here is an
example. Let $X=\Spec(k[x]/(x^2))$ with action of
$\mu_n=\Spec(k[z]/(z^n-1))$ by multiplication then $Y=X/R=\Spec(k)$.
We have $X\times_Y X=\Spec(k[x_1,x_2]/(x_1^2,x_2^2))$. The morphism
$j_Y:R\to X\times_Y X$ is given by the map of $k$-algebras
$k[x_1,x_2]/(x_1^2,x_2^2)\to k[x,z]/(x^2,z^n-1)$ such that
$x_1\mapsto x$ and $x_2\mapsto zx$. The element $x_1x_2$ is
not zero and it is mapped to $zx^2=0$.

\smallskip

\no (2) The map $X\to X/R$ need not be of finite type even when
$R\rightrightarrows X$ is finite locally free. For example if
$X=\Spec(k[t_1,t_2,\dots])$ with action of $\mu_n$
by $z.t_i=zt_i$ then $X/R$ is the spectrum of the ring of polynomials
all whose homogeneous components have degree a multiple of $n$.
\end{remarks}

In the rest of the text, we will focus on flat groupoids such
that the morphism $j_Y:R\to X\times_Y X$ is an epimorphism.
%This happens in most situations, in spite of cases
%like~\ref{examples of quotient}(1).
This occurs for instance when $X\to Y$ is flat and there is
a schematically dense open subscheme $X_0\subset X$ where
the action is free. One way to measure further the good
behavior of these groupoids is furnished
by Proposition~\ref{prop:canonical_factorization}
and leads to the following notion.

\begin{definition}
Let $R\rightrightarrows X$ be a flat, locally finitely presented groupoid space
with finite stabilizer. We say that
{\em $R\rightrightarrows X$ has complexity $n$}
if the map $j_Y:R\to X\times_Y X$ is an epimorphism
and the length of its canonical sequence is $n$.
\end{definition}

\begin{remarks}
(1) The groupoid $R\rightrightarrows X$ has complexity $0$ if and only if it
is free. It has complexity at most $1$ if and only if $j_Y$ is
an effective epimorphism.

\smallskip

\noindent (2) If $j_Y$ is an epimorphism, then,
by Proposition~\ref{prop:canonical_factorization}(5),
a sufficient condition for a groupoid to
have finite complexity is that $X$ is of finite type over
a fixed noetherian base scheme.

\smallskip

\noindent (3) Levelt's
results \cite{Le65}, see Example~\ref{examples:effective_ring_extensions}(4),
hint
that finite locally free groupoids with isolated fixed points of
stabilizer degree at most $2$ (e.g., an action of a group scheme
of order 2 with isolated fixed points) should have complexity
at most $1$. We shall see examples of this in the next
section.
\end{remarks}

\subsection{Examples} \label{examples}

Because the formation of the canonical sequence
is local on the base for the flat topology
(Proposition~\ref{prop:canonical_factorization}(2)),
the computation of the complexity can be done locally.
It follows that computations in this section provide
results also for groupoids which are group actions only
locally for the flat topology, or locally after passage
to a completed local ring. This applies for instance to
quotients of surfaces by $p$-closed vector fields, studied
by many people in the last 40 years (Rudakov--Shafarevich,
Russell, Ekedahl, Katsura--Takeda, Hirokado...).

We start with examples valid in any characteristic.

\begin{prop}
Let $X=\AA^n_S$ be affine $n$-space over a scheme $S$.
Let $G$ be the symmetric group on~$n$ letters, acting by
permutation of the coordinates of $X$. Then the quotient
map $\pi:X\to Y=X/G$ is finite locally free of rank $n!$
and the groupoid $G\times X\rightrightarrows X$ has complexity 1.
\end{prop}

\begin{proof}
First we set the notations. We may assume $S=\Spec(R)$ affine.
Then $X=\Spec(B)$ where $B=R[x_1,\dots,x_n]$ is a polynomial
ring in $n$ variables, and $Y=\Spec(A)$ where $A=B^G$ is the
ring of invariants. Let $S_k(X_1,\dots,X_n)$ be the symmetric
function of degree $k$ in $X_1,\dots,X_n$ and
$s_k=S_k(x_1,\dots,x_n)\in A$. By the Main Theorem on symmetric
functions, we have $A=R[s_1,\dots,s_n]$ which is a ring of
polynomials in the variables $s_i$, moreover
\[
B\simeq
\frac{A[x_1,\dots,x_n]}{(S_1(x_i)-s_1,\dots,S_n(x_i)-s_n)}
\]
and therefore
\[
B\otimes_A B\simeq
\frac{B[X_1,\dots,X_n]}{(S_1(X_i)-s_1,\dots,S_n(X_i)-s_n)}
\]
is $B$-free of rank $n!$ with basis the set of monomials
$\sB=\{X_1^{d_1}\dots X_n^{d_n};0\le d_i<i,\forall i\}$.
The map $j:G\times X\to X\times_YX$ corresponds to the map
of $B$-algebras which is given by evaluation on
$(x_1,\dots,x_n)$ and its permutations:
\[
\ev:
\frac{B[X_1,\dots,X_n]}{(S_1(X_i)-s_1,\dots,S_n(X_i)-s_n)}
\too \prod_{\sigma\in \mathfrak{S}_n} B \quad,\quad
P\longmapsto (P(x_{\sigma(1)},\dots,x_{\sigma(n)}))_{\sigma\in \mathfrak{S}_n}.
\]
The stabilizer $\Sigma\to X$ of the groupoid has function ring:
\[
B[\Sigma]=
\prod_{\tau\in \mathfrak{S}_n} \frac{B}{(x_1-x_{\tau(1)},\dots,x_n-x_{\tau(n)})}.
\]
The two maps $\pr_2,d:\Sigma\times_X(G\times X)\rightrightarrows G\times X$
correspond to the maps of $B$-algebras
\[
\alpha,\beta:
\prod_{\sigma\in \mathfrak{S}_n} B \too \prod_{\sigma,\tau\in \mathfrak{S}_n}
\frac{B}{(x_1-x_{\tau(1)},\dots,x_n-x_{\tau(n)})}
\]
defined by $\alpha(Q)_{\sigma,\tau}=Q_\sigma$ and
$\beta(Q)_{\sigma,\tau}=Q_{\tau\sigma}$
for all $Q=(Q_\sigma)_{\sigma\in \mathfrak{S}_n}$.

Since the action of $G$ on $X$ is not free, the complexity
of the groupoid is not 0. Hence what remains to be proved
is that $\ev$ is injective and $\im(\ev)=\ker(\alpha-\beta)$.
In order to describe the image of $\ev$ let us introduce some
more notation. Let $E$ be the set of pairs of integers $(i,j)$
with $1\le i<j\le n$. Let
$V=V(x_1,\dots,x_n)=\prod_{(i,j)\in E}(x_j-x_i)$ be the
Vandermonde of the $x_i$. To each subset $F\subset E$ we
attach a monomial $\mu(F)=\prod_{(i,j)\in F} X_j$.
For example if $n=4$ and $F=\{(1,3),(2,4),(3,4)\}$ then
$\mu(F)=X_3X_4^2$. Obviously the map
$\mu:\sP(E)\to \sB$ is surjective and if $M=\mu(F)$ then $\deg(M)=\card(F)$.
%The fiber of~$\mu$ above a monomial $M=X_1^{d_1}\dots X_n^{d_n}$
%is described by the bijection:
%\[
%\prod_{j=2}^n\sP_{d_j}(\{1,\dots,j-1\})\isomto \mu^{-1}(M)
%\quad,\quad
%(G_2,\dots,G_n)\mapsto \bigcup_{j=2}^n G_j\times \{j\}.
%\]
Now for each basis monomial $M\in \sB$ we define a $B$-linear
form $\varphi_M:\prod_{\sigma\in \mathfrak{S}_n} B\to B$ by
\[
Q=(Q_\sigma)_{\sigma\in \mathfrak{S}_n} \longmapsto
\varphi_M(Q)=\sum_{\sigma}\varepsilon(\sigma)
\bigg(\sum_{F\subset E\atop\mu(F)=M} \prod_{(i,j)\in E-F} x_{\sigma(i)}\bigg)
Q_\sigma.
\]
(Here $\varepsilon(\sigma)$ is the sign of the permutation $\sigma$.)
We let $\varphi:\prod_{\sigma\in \mathfrak{S}_n} B\to \prod_{M\in\sB} B$ be
the map with components $\varphi_M$
and we use the same letter to denote the map with values
in $\prod_{M\in\sB} B/VB$ obtained by reduction mod $V$.
We claim that the following sequence is exact:
\[
0\too
\frac{B[X_1,\dots,X_n]}{(S_1(X_i)-s_1,\dots,S_n(X_i)-s_n)}
\stackrel{\ev}{\too} \prod_{\sigma\in \mathfrak{S}_n} B
\stackrel{\varphi}{\too} \prod_{M\in\sB} B/VB.
\]
In order to prove this we introduce suitable Lagrange
interpolation polynomials which allow us to invert the map
$\ev$ after the base change $B\to B[1/V]$. Precisely, we set:
\[
L_\sigma(X_1,\dots,X_n)=\frac{\varepsilon(\sigma)}{V}
\prod_{(i,j)\in E} (X_j-x_{\sigma(i)}).
\]
We have $\deg_{X_i}(L_\sigma)<i$ for all $i=1,\dots,n$.
Thus, after inverting $V$, the polynomial $L_\sigma$ lies in the submodule
$\oplus_{M\in\sB} R\cdot M \subset B[X_1,\dots,X_n]$ which
as we said earlier maps isomorphically onto
$B[X_1,\dots,X_n]/(S_1(X_i)-s_1,\dots,S_n(X_i)-s_n)$.
Moreover one sees that $L_\sigma(x_{\tau(1)},\dots,x_{\tau(n)})=
\delta_{\sigma,\tau}$ (Kronecker $\delta$). From these remarks
follows that the inverse to
$\ev\otimes \id_{B[1/V]}$ is given by interpolation, that is:
\[
\interp(Q)=\sum_{\sigma\in \mathfrak{S}_n} Q_\sigma L_\sigma.
\]
From this, since $V$ is a nonzerodivisor in $B$, the injectivity
of $\ev$ follows. By expanding one finds:
\begin{align*}
\interp(Q)
& =  \frac{1}{V}
\sum_{\sigma\in \mathfrak{S}_n} \varepsilon(\sigma)Q_\sigma
\prod_{(i,j)\in E} (X_j-x_{\sigma(i)}) \\
& =  \frac{1}{V}
\sum_{\sigma\in \mathfrak{S}_n} \varepsilon(\sigma)Q_\sigma
\sum_{F\subset E} (-1)^{\card(E-F)}
\cdot\prod_{(i,j)\in E-F} x_{\sigma(i)} \cdot\mu(F) \\
& =  \frac{1}{V}
\sum_{\sigma\in \mathfrak{S}_n} \varepsilon(\sigma)Q_\sigma
\sum_{M\in \sB}\sum_{F\subset E\atop\mu(F)=M}
(-1)^{\frac{n(n-1)}{2}-\deg(M)}
\prod_{(i,j)\in E-F} x_{\sigma(i)} \cdot M \\
& =  \frac{1}{V}\sum_{M\in \sB}(-1)^{\frac{n(n-1)}{2}-\deg(M)}
\varphi_M(Q) \cdot M.
\end{align*}
Since $Q=(\ev\otimes \id_{B[1/V]})(\interp(Q))$, 
we see that $Q$ lies in the image of $\ev$ if and only if the
components of $\interp(Q)$ on the basis vectors $M\in \sB$
lie in $B$. This means precisely that $\varphi_M(Q)$ is divisible by $V$
for all $M\in\sB$, which proves the exactness of the sequence.

We can now conclude. It is clear that $\im(\ev)\subset\ker(\alpha-\beta)$.
In order to prove the reverse inclusion let
$Q=(Q_\sigma)_{\sigma\in \mathfrak{S}_n}$ lie in the equalizer of
$\alpha$ and $\beta$, that is:
\[
Q_{\tau\sigma}\equiv Q_\sigma\mod
(x_1-x_{\tau(1)},\dots,x_n-x_{\tau(n)}) , \quad
\mbox{for all } \sigma,\tau\in \mathfrak{S}_n.
\]
We want to prove that $\varphi_M(Q)$ is divisible by $V$
for all $M\in\sB$. It is enough to prove that $\varphi_M(Q)$
is divisible by $x_v-x_u$ for all $(u,v)\in E$. Consider the
transposition $\tau=(u,v)$. Then $\mathfrak{S}_n$ is partitioned into $n!/2$
pairs $\{\sigma,\tau\sigma\}$ and it is enough to prove that
for each $\sigma$ the sum
\[
\varepsilon(\sigma)
\bigg(\sum_{F\subset E\atop\mu(F)=M} \prod_{(i,j)\in E-F} x_{\sigma(i)}\bigg)
Q_\sigma
+
\varepsilon(\tau\sigma)
\bigg(\sum_{F\subset E\atop\mu(F)=M} \prod_{(i,j)\in E-F} x_{\tau\sigma(i)}\bigg)
Q_{\tau\sigma}
\]
is divisible by $x_v-x_u$. This is clear, because modulo $x_v-x_u$
we have $Q_{\tau\sigma}\equiv Q_\sigma$ by the assumption on~$Q$
and $x_{\tau\sigma(i)}\equiv x_{\sigma(i)}$ by the definition of $\tau$.
\end{proof}

\begin{remark}
More generally, we can ask if the complexity is at most 1
for a finite constant group $G$ acting on a smooth scheme
$X$ in such a way that the pointwise stabilizers $G_x$ are
generated by reflections, in the sense that there is a system
of local coordinates such that $G_x$ is generated by
linear automorphisms of order 2.
\end{remark}

Here is another example in arbitrary characteristic.

\begin{lemma}
Let $R$ be a ring. Let $n\ge 2$ be an integer.
Let $X=\AA^1_R$ be the affine line over $R$,
with the action of $G=\mu_{n,R}$ given by $G\times X\to X$,
$(z,x)\mapsto zx$. Then the groupoid $G\times X\rightrightarrows X$ has
complexity $1$ if $n=2$ and at least~$2$ otherwise.
If $n=3$, the complexity is equal to $2$.
\end{lemma}

\begin{proof}
We have $X=\Spec(B)$ and $Y=X/G=\Spec(A)$ with $B=R[x]$,
$A=R[y]$ and $y=x^n$. Let $C_\infty=B\otimes_AB=B[X]/(X^n-x^n)$
and $C_0=B[z]/(z^n-1)$. The question is about the finite
morphism of $B$-algebras $\rho:C_\infty\to C_0$ with
$\rho(X)=zx$. Note that $\rho$ identifies $C_\infty$
with the sub-$B$-algebra of $B[z]/(z^n-1)$ generated by $zx$.
We have
$C_0\otimes_{C_\infty} C_0=B[z_1,z_2]/(z_1^n-1,z_2^n-1,x(z_1-z_2))$
with the maps $\alpha,\beta:C_0\to C_0\otimes_{C_\infty} C_0$
given by $\alpha(z)=z_1$ and $\beta(z)=z_2$. Let
$C_1\subset C_0$ be the equalizer
of these maps, this is the sub-$B$-algebra
generated by the elements $y_i:=z^ix$ for $i=1,\dots,n-1$.
If $n=2$ we have $C_\infty=C_1$, so the complexity is 1.
If $n\ge 3$ we have $z^2x\in C_1\setminus C_\infty$ and the
complexity is at least~2. In general $C_1$ has a fairly
complicated structure. We leave it to the reader to check
that for $n=3$ we have
$C_1=B[y_1,y_2]/(y_1^3-x^3,y_1y_2-x^2,y_2^2-xy_1)$
and that the map $C_\infty\to C_1$ is effective.
\end{proof}

Finally an example in characteristic $p$.

\begin{lemma}
Let $R$ be a ring of characteristic $p>0$. Let $X=\AA^1_R$ be
the affine line over $R$, with the action of $G=\alpha_{p,R}$
given by $G\times X\to X$, $(a,x)\mapsto \frac{x}{1+ax}$. Then
the groupoid $G\times X\rightrightarrows X$ has complexity~$1$ if $p=2$
and at least~$2$ otherwise.
\end{lemma}

\begin{proof}
We have $X=\Spec(B)$ and $Y=X/G=\Spec(A)$ with $B=R[x]$,
$A=R[y]$ and $y=x^p$. The question is about exactness of the
sequence of $B$-algebras:
\[
\xymatrix{
\frac{B[X]}{X^p-x^p} \ar[r]^-\rho &
\frac{B[a]}{a^p} \ar@<1mm>[r]^-\alpha \ar@<-1mm>[r]_-\beta &
\frac{B[a_1,a_2]}{a_1^p,a_2^p,x^2(a_1-a_2)}}
\]
with $\rho(X)=\frac{x}{1+ax}$, $\alpha(a)=a_1$,
$\beta(a)=a_2$. In order to find the image of $\rho$ we
compute in the localizations with respect to $x$. Since $\rho$
is injective we write $X$ for $\rho(X)$. From
$X=\frac{x}{1+ax}$ we get $a=X^{-1}-x^{-1}$ so if
$Q(a)=\sum_{i=0}^{p-1}Q_ia^i$ is the image of some $P$
under $\rho$ then we have:
\[
P(X)=Q(X^{-1}-x^{-1})=
\sum_{i=0}^{p-1}(-1)^ix^{-i}Q_i
+\sum_{j=1}^{p-1}
\left(\sum_{i=j}^{p-1}(-1)^{i-j}{i\choose j}
x^{-p-i+j}Q_i\right)X^{p-j}.
\]
We find that the image of $\rho$ is the set of $Q$ such that
$x^{p-1}$ divides $\sum_{i=1}^{p-1}(-1)^ix^{p-1-i}Q_i$
and $x^{2p-1-j}$ divides $\sum_{i=j}^{p-1}(-1)^i{i\choose j}
x^{p-1-i}Q_i$ for all $j=1,\dots,p-1$. This may be rewritten
as the set of $Q$ such that $x^{i+1}$ divides $Q_i$ for all
$i=1,\dots,p-1$ (say $Q_i=x^{i+1}R_i$ for some $R_i\in B$)
and $x^{p-1-j}$ divides
$\sum_{i=j}^{p-1}(-1)^i{i\choose j}R_i$ for all $j=1,\dots,p-1$.
On the other hand,
the equalizer of $\alpha$ and $\beta$
is the set of $Q$ such that $x^2$ divides $Q_i$ for all $i=1,\dots,p-1$.
These sets
are equal if and only if $p=2$.
\end{proof}

\section{Main theorems}
\label{section:main theorems}

After the work of the previous sections, we are ready to give
an answer to the descent question from the introduction, for
groupoids of complexity at most $1$. It applies to the objects of
a stack whose isomorphism sheaves are representable: see
Theorem~\ref{theorem:descent_along_quotient}.

\subsection{Equivariant objects} \label{ss:equivariant objects}

\begin{definition}
Let $s,t:R\rightrightarrows X$ be a groupoid and
$c,\pr_1,\pr_2:R\times_{s,X,t}R\to R$ the composition and projections.
Let $\cC\to \AlgSp$ be a category fibered over the category of algebraic
spaces and let $\cF\in\cC(X)$ be an object.
An {\em $R$-linearization} on $\cF$
is an isomorphism $\phi:s^*\cF\isomto t^*\cF$ satisfying the cocycle
condition $c^*\phi=(\pr_1^*\phi)\circ(\pr_2^*\phi)$,
meaning that the following triangle is commutative~:
\begin{center}
\begin{tikzpicture}[scale=1.3,text height=1.5ex, text depth=0.5ex]
\node (A1) at (-0.5,0) {$(s\pr_2)^*\cF$};
\node (A2) at (0.7,0) {$=(sc)^*\cF$};
\node (B1) at (3.5,0) {$(tc)^*\cF=$};
\node (B2) at (4.65,0) {$(t\pr_1)^*\cF$};
\node (C1) at (1.3,-2) {$(t\pr_2)^*\cF$};
%\node (C2) at (2,-2) {$=$};
\node (C3) at (2.7,-2) {$=(s\pr_1)^*\cF.$};
\draw[-myto,font=\scriptsize] (A2) edge node[pos=.55,above] {$c^*\phi$} (B1);
\draw[-myto,font=\scriptsize] (A1) edge node[pos=.3,below=2mm] {$\pr_2^*\phi$} (C1);
\draw[-myto,font=\scriptsize] (C3) edge node[pos=.7,below=2mm] {$\pr_1^*\phi$} (B2);
%\path[-myto,font=\scriptsize]
%(A) edge node[pos=.55,above] {$\pr_1^*\phi$} (B)
%(A) edge node[pos=.25,below=2mm] {$c^*\phi$} (C)
%(B) edge node[pos=.25,below=2mm] {$\pr_2^*\phi$} (C);
\end{tikzpicture}
\end{center}
An {\em $R$-equivariant object of $\cC$ over $X$} is an object
$\cF\in\cC(X)$ together with an $R$-linearization. We write
$\cC(R,X)$ for the category of $R$-equivariant objects.
\end{definition}

%\begin{remark}
%For quasi-coherent sheaves this definition is given in
%\cite[\spref{03LH}]{SP}.
%\end{remark}

\begin{exam} \label{example_linearization}
Let $R\rightrightarrows X$ be a groupoid as above and let $\pi:X\to Y$ be
an $R$-invariant morphism, i.e., $\pi s=\pi t$. Then for any
object $\cG\in \cC(Y)$, the pullback $\cF=\pi^*\cG$ is endowed
with a canonical $R$-linearization
$\phi:s^*\cF=s^*\pi^*\cG \simeq (\pi s)^*\cG=
(\pi t)^*\cG \simeq t^*\pi^*\cG=t^*\cF$.
\end{exam}

We recall the notion of a {\em square}, which is closely related
to that of $R$-equivariant object.

\begin{definition}
A morphism of groupoids $f:(R',X')\to (R,X)$ is called a {\em square} or
{\em cartesian}
when the commutative diagram
\[
\xymatrix{R' \ar[r] \ar[d]_f & X' \ar[d]^f \\
R \ar[r] & X}
\]
is cartesian, if we take for horizontal maps either both source maps,
or both target maps.
\end{definition}

To illustrate these definitions, take for $\cC$ the category of algebraic spaces
over algebraic spaces. For $(X'\to X)\in\cC(X)$, the following lemma makes it clear
that an $R$-linearization on
$X'$ is the same as a lift of the $R$-action to $X'$.

\begin{lemma}\label{lemma:square}
Let $s,t:R\rightrightarrows X$ be a groupoid. Let $(f:X'\to X,\phi:s^*X'\isomto t^*X')$
be an $R$-equivariant $X$-space.
Complete $X'$ to a quintuple $(R',X',s',t',c')$ as follows~:
\begin{trivlist}
\itemn{1} $R'=s^*X'=R\times_{s,X,f}X'$ whose $T$-points are pairs $(\alpha,x')$
with $\alpha\in R(T)$ and $x'\in X'(T)$,
\itemn{2} $s'=\pr_2:R'\to X'$,
\itemn{3} $t'=\pr_2\circ\phi:R\times_{s,X,f}X'\too R\times_{t,X,f}X'\too X'$,
\itemn{4} $c':R'\times_{s',X',t'}R'\too R'$ defined on $T$-points
by $c'\bigl((\alpha,x'),(\beta,y')\bigr)=(\alpha\beta,y')$.
\end{trivlist}
Then $(R',X',s',t',c')$ is a groupoid and the morphism
$(R',X')\to (R,X)$ is a square morphism of groupoids.

Conversely, a square morphism of groupoids $(R',X')\to (R,X)$ gives an
$R$-equivariant $X$-space $(X'\to X,s^*X'\isomto R'\isomto t^*X')$.
\end{lemma}

\begin{proof}
This is \cite[\spref{0APC}]{SP}.
%Note that the composition $\alpha\beta=c(\alpha,\beta)$ in the
%definition of $c'$ makes sense because $ft'=t\pr_1$, so that given
%two $T$-points $(\alpha,x')$ and $(\beta,y')$ in $R'(T)=R\times_{s,X,f}X'(T)$
%such that $x'=s'(\alpha,x')=t'(\beta,y')$, we have
%$s(\alpha)=f(x')=ft'(\beta,y')=t(\beta)$.
\end{proof}

\subsection{Descent along the quotient}

Let $s,t:R\rightrightarrows X$ be a flat locally finitely presented groupoid. In this
section we are interested in the problem of descending objects
of a category $\cC$ fibered over the category of algebraic spaces
along the quotient map $\pi:X\to X/R=Y$. We know that for any
object $\cG\in \cC(Y)$, the pullback $\cF=\pi^*\cG$ is endowed
with a canonical $R$-linearization (example~\ref{example_linearization}).
Conversely, if $\cF\in \cC(X)$ then the datum of an
$R$-linearization allows to descend~$\cF$ to an object based on $[X/R]$,
the quotient {\em as an algebraic stack}, but is not enough to
descend~$\cF$ to an object of~$\cC(Y)$ in general.
Let $\cC(R,X)$ be the category of $R$-equivariant objects
$(\cF,\phi)$. Descent Theory as formulated by Grothendieck seeks
to characterize the essential image
of the pullback functor $\pi^*:\cC(Y)\to \cC(R,X)$.
When $\cC$ is the category of \'etale morphisms of spaces,
and without additional conditions on $R\rightrightarrows X\to Y$, Keel and
Mori~\cite[Lem.~6.3]{KM97}, Koll\'ar~\cite[\S~2]{Ko97},
Rydh~\cite[\S~3]{Ry13} obtain such a characterization in terms
of fixed-point reflecting $R$-equivariant objects. In a different
direction, we shall prove that if $R\rightrightarrows X$ has complexity at
most $1$ and flat quotient $X\to Y$, there is a nice description
of the image of $\pi^*$ for very general stacks $\cC$.

\begin{definition}
Let $\Sigma=\Stab_R$ be the stabilizer of the groupoid, let
$a:\Sigma\to R$ be the inclusion, and put $b=sa=ta$.
We denote by $\cC(R,X)^{\Sigma}$ the full
subcategory of $\cC(R,X)$ consisting of $R$-equivariant objects
$(\cF,\phi)$ such that the action of $\Sigma$ is trivial, meaning
that the following map is the identity:
\[
b^*\cF\simeq a^*s^*\cF \stackrel{a^*\phi}{\tooo}a^*t^*\cF\simeq b^*\cF.
\]
\end{definition}

To dispel the dryness of the formalism of groupoids, we emphasize
that if $\cC$ is the category of schemes or algebraic spaces, and
if the groupoid is given by the action of a group~$G$, then a
$G\times X$-linearization on some $X'\in \cC(X)$ is equivalent to a lift of
the action of $G$ to $X'$ and the action of $\Sigma$ is trivial in
the above sense if and only if it is trivial in the usual sense.

\begin{lemma}
The functor $\pi^*:\cC(Y)\to \cC(R,X)$
takes values in $\cC(R,X)^{\Sigma}$.
\end{lemma}

\begin{proof}
We have to show that the canonical $R$-linearization of a pullback
$\cF=\pi^*\cG$ becomes trivial when restricted to $\Sigma$. Recall
from \cite[A.1]{Gr59} or \cite[\spref{003N}]{SP},
that in a fibered category, there are isomorphisms
$(fg)^*\isomto g^*f^*$ between pullback functors, and commutative
squares giving compatibility for triple compositions~:
\[
\xymatrix{
(fgh)^* \ar[r] \ar[d] & (gh)^*f^* \ar[d] \\
h^*(fg)^* \ar[r] & h^*g^*f^*.}
\]
We write the two squares picturing such compatibility for the two
compositions $\pi sa:\Sigma\to Y$ and $\pi ta:\Sigma\to Y$, taking
advantage of the fact that $\pi s=\pi t$ in order to glue them
on one side:
\[
\xymatrix@C=20mm{
(sa)^*\pi^* \ar[d] %\ar[d]_-{\alpha_{a,s}\star \id_{\pi^*}}
\ar@/^2.5pc/[rr]^-{\id}
& (\pi sa)^*=(\pi ta)^*
\ar[l] %\ar[l]_-{\alpha_{sa,\pi}}
\ar[r] %\ar[r]^-{\alpha_{ta,\pi}}
\ar[d]
& (ta)^*\pi^*
\ar[d] \\ %\ar[d]^-{\alpha_{a,t}\star \id_{\pi^*}}
a^*s^*\pi^*
\ar@/_2.5pc/[rr]_-{a^*\phi}
& a^*(\pi s)^*=a^*(\pi t)^*
\ar[l] %\ar[l]_-{\id_{a^*}\star\alpha_{s,\pi}}
\ar[r] %\ar[r]^-{\id_{a^*}\star\alpha_{t,\pi}}
& a^*t^*\pi^*.}
\]
Since $sa=ta$ we see that the top row is the identity.
The commutativity of the exterior diagram is exactly the claim
we want to prove.
\end{proof}

\begin{theorem} \label{theorem:descent_along_quotient}
Let $R\rightrightarrows X$ be a flat, locally finitely presented groupoid space
with finite stabilizer $\Sigma\to X$ and complexity at most $1$.
Assume that the quotient $\pi:X\to Y=X/R$ is flat (resp.\ flat
and locally of finite presentation). Let $\cC\to\AlgSp$ be a stack in
categories for the fpqc topology (resp.\ for the fppf topology).
\begin{trivlist}
\itemn{1} If the sheaves of homomorphisms $\cH om_{\cC}(\cF,\cG)$
have diagonals which are representable by algebraic spaces, then
the pullback functor $\pi^*:\cC(Y)\to \cC(R,X)^{\Sigma}$ is fully faithful.
\itemn{2} If the sheaves of isomorphisms $\cI som_{\cC}(\cF,\cG)$ 
are representable by algebraic spaces, then the pullback functor
$\pi^*:\cC(Y)\to \cC(R,X)^{\Sigma}$ is essentially surjective.
\end{trivlist}
In particular if $\cC$ is a stack in groupoids with representable
diagonal, the functor $\pi^*$ is an equivalence.
\end{theorem}

In Section~\ref{examples} many examples were given that satisfy
the assumptions of the Theorem.

\begin{remark} \label{remark:philosophy}
This result is not really an alternative to faithfully flat descent,
but rather a refinement of it. Indeed, faithfully flat descent
{\em does} provide an answer to the question of the image of~$\pi^*$:
it is the particular case of our theorem for the flat groupoid
$R_1:=X\times_Y X\rightrightarrows X$ whose stabilizer is trivial. The
category $\cC(R_1,X)$ comprises objects with descent data, the
latter being isomorphisms on products $X\times_YX$ with conditions
on triple products $X\times_YX\times_YX$. However, it is often the
case in concrete geometric situations that there is a natural action
of a group or groupoid $R\ne R_1$ such that it is much easier to handle
$R$-equivariant objects. In these situations, the functor of points
of the quotient $Y=X/R$ is usually hard to describe, as well as the
square and the cube of $X$ over $Y$, making $\cC(R_1,X)$ less
convenient.
\end{remark}

\begin{proof}
The assumptions on $\cC$ and $\pi$ imply that effective descent
along $\pi$ holds in $\cC$~;
in the fpqc case note that $\pi$ is an fpqc covering since
it is open (\ref{theo:quotient_by_ff_groupoid}) and faithfully
flat, see e.g.\ Vistoli~\cite[Prop.~2.35]{Vi05}.
Since the map $j_Y:R\to X\times_Y X$
will come up repeatedly, we write simply $j:=j_Y$.

\smallskip

\noindent (1) Let $\cG_1,\cG_2\in\cC(Y)$ and let
$(\cF_1,\phi_1),(\cF_2,\phi_2)\in \cC(R,X)^{\Sigma}$ be their
pullbacks. We must prove that the map:
\[
\Hom_{\cC(Y)}(\cG_1,\cG_2)\too
\Hom_{\cC(R,X)^{\Sigma}}\bigl((\cF_1,\phi_1),(\cF_2,\phi_2)\bigr)
\]
is bijective.
Injectivity is a consequence of the fact that $\pi:X\to Y$
is a covering for the topology for which $\cC$ is a stack,
and the fact that $\Hom_{\cC(Y)}(\cG_1,\cG_2)$ is a separated
presheaf. For surjectivity let
$f:(\cF_1,\phi_1)\to(\cF_2,\phi_2)$
be a morphism. Let $\pi_1,\pi_2:X\times_YX\to X$ be the
projections. By descent it is enough to prove that
$\pi_1^*f=\pi_2^*f$. By construction $\phi_i$ is the identity
of $q^*\sG_i$ for $i=1,2$, where $q=\pi s=\pi t$. Therefore
$s^*f=t^*f$. Write $H:=\cH om_{\cC(Y)}(\cG_1,\cG_2)$.
We have a commutative diagram:
\[
\xymatrix@C=15mm{
R \ar[r]^-j \ar[d] & X\times_YX \ar[d]^-d \\
H \ar[r]^-{\Delta} & H\times_Y H}
\]
where $d:=(\pi_1^*f,\pi_2^*f)$.
Since the diagonal $\Delta$ is assumed to be representable,
the fiber product
\[P:=H\times_{(\Delta,d)} X\times_YX\]
is representable and the map $j$ factors through a map
$k:P \to X\times_YX$. Since the groupoid has complexity
at most $1$, the map $j$ is an effective
epimorphism. It follows by formal arguments that $k$ has
the same property. Since $k$ is a pullback of the
diagonal, it is also a monomorphism. Thus, $k$ is an
isomorphism, and therefore $\pi_1^*f=\pi_2^*f$.

\smallskip

\noindent (2) Let
$(\cF,\phi)\in \cC(R,X)^{\Sigma}$ be an $R$-equivariant object.
Given that $R\times_{X\times_Y X} R$ is isomorphic
to $\Sigma\times_{(s,t)} R$ via the map
$(\varphi,\psi)\mapsto (\varphi\psi^{-1},\psi)$, the exact
sequence for the effective epimorphism $j$ is:
\[
\xymatrix@C=10mm{
\Sigma\times_{(s,t)} R \ar@<.5ex>[r]^-d\ar@<-.5ex>[r]_-{\pr_2} & R \ar[r]^-{j} & X\times_Y X.}
\]
Here $d$ is the composition
$\xymatrix@C=8mm{
\Sigma\times_{(s,t)} R \ar[r]^-{a\times\id}
& R\times_{(s,t)} R \ar[r]^-c & R}$.
It follows that for all $X\times_Y X$-algebraic spaces $I$,
we have an exact diagram of sets:
\[
\xymatrix{
\Hom(X\times_Y X,I) \ar[r]^-{j^*} & \Hom(R,I) \ar@<.5ex>[r]^-{d^*} \ar@<-.5ex>[r]_-{\pr_2^*} &
\Hom(\Sigma\times_{(s,t)} R,I).}
\]
Let $\pi_1,\pi_2:X\times_Y X\to X$ be the projections, and let
$I=\cI som_{X\times_Y X}(\pi_2^*\cF,\pi_1^*\cF)$. This is an algebraic
space by assumption, so from the above we obtain an exact diagram of sets:
\[
\xymatrix@C=10mm{
\Isom_{X\times_Y X}(\pi_2^*\cF,\pi_1^*\cF) \ar[r]^-{j^*} &
\Isom_R(s^*\cF,t^*\cF) \ar@<.5ex>[r]^-{d^*} \ar@<-.5ex>[r]_-{\pr_2^*} &
\Isom_{\Sigma\times R}(\pr_2^*s^*\cF,\pr_2^*t^*\cF).}
\]
Here we use the identifications $d^*s^*\cF\simeq (sd)^*\cF=(s\pr_2)^*\cF\simeq
\pr_2^*s^*\cF$ which need no further comment, and the similar
identifications with $s$ replaced by~$t$ which require the
observation that $td=t\pr_2$ since source and target agree on the
stabilizer. Now consider the cocycle condition
$c^*\phi=\pr_1^*\phi\circ\pr_2^*\phi$ on $R\times_{(s,t)} R$ satisfied by the
$R$-linearization $\phi:s^*\cF\to t^*\cF$. Then after pullback along
$a\times\id:\Sigma\times_{(s,t)} R\to R\times_{(s,t)} R$,
and since the stabilizer acts trivially on~$\cF$, this becomes:
\[
d^*\phi=(a\pr_1)^*\phi\circ\pr_2^*\phi=\pr_2^*\phi.
\]
Therefore by exactness of the diagram of $\Isom$ sets, $\phi$ descends to an
isomorphism $\psi:\pi_2^*\cF\isomto \pi_1^*\cF$. To conclude,
we use descent along the map $\pi:X\to Y$. For~$\psi$ to be
a descent datum with respect to $X\to Y$, it need only satisfy
the usual gluing condition:
\[
(\star) \quad \pi_{13}^*\psi=\pi_{12}^*\psi\circ\pi_{23}^*\psi
\]
where $\pi_{ij}:X\times_Y X\times_YX\to X\times_Y X$ are the projections.
In order to prove that this indeed holds, we
consider the commutative diagram:
\[
\xymatrix@C=15mm{
R\times_{s,X,t}R \ar[r]^-{j\times j} \ar@<1ex>[d] \ar@<-1ex>[d]_{\pr_1,\pr_2,c} \ar[d] &
X\times_Y X\times_YX
 \ar@<1ex>[d]^{\pi_{12},\pi_{23},\pi_{13}} \ar@<-1ex>[d] \ar[d] \\
R \ar[r]^j & X\times_YX.}
\]
On pulling back the relation $(\star)$ by $j\times j$ we obtain
the relation $c^*\phi=(\pr_1^*\phi)\circ (\pr_2^*\phi)$ which holds
by assumption. Since $X\to Y$ is flat, the morphism
$j\times j$ is finite, surjective and schematically dominant, hence
an epimorphism. Therefore Condition $(\star)$ holds, hence by
descent~$\cF$ is the pullback of an object $\cG\in\cC(Y)$.
\end{proof}

\begin{theorem} \label{theorem:descent_of_flat}
Let $\cC\to\AlgSp$ be one of the following stacks in categories:
\begin{trivlist}
\itemn{1} $\cC_1=\Flat$, the fppf stack whose objects over $X$ are flat
morphisms of algebraic spaces $X'\to X$.
\itemn{2} $\cC_2=\Flat^{\qa}$, the fpqc stack whose objects over $X$ are
quasi-affine flat morphisms of algebraic spaces $X'\to X$.
\end{trivlist}
Let $R\rightrightarrows X$ be a flat, locally finitely presented groupoid space
with finite stabilizer $\Sigma\to X$ and complexity at most $1$.
Assume that the quotient $\pi:X\to Y=X/R$ is flat and locally
of finite presentation if $\cC=\cC_1$, and flat if $\cC=\cC_2$.
Then the functor
$\pi^*:\cC(Y)\to \cC(R,X)^{\Sigma}$ is an equivalence.
\end{theorem}

Recall that an object $(X'\to X)\in \cC(X)$ is equivalent to a flat square
morphism of groupoids $(R',X')\to (R,X)$ (Lemma~\ref{lemma:square}) and when
$(X'\to X)=\pi^*(Y'\to Y)$, then $Y'=X'/R'$. The essentially surjectivity of
$\pi^*$ can thus be rephrased as: the natural morphism $X'\to (X'/R')\times_Y
X$ is an isomorphism, and $X'/R'\to Y$ is flat.

\begin{proof}
%In all cases the morphism $\pi:X\to Y$ is of effective
%descent in $\cC$, by \cite[\spref{04UC}]{SP},
Here the conditions on the representability of the diagonal
of $\cC$ fail to hold, so we need different arguments.
In order to prove full faithfulness let $W_1,W_2$ be
objects of $\cC(Y)$ and $(V_1,\phi_1)$, $(V_2,\phi_2)$ the
pullbacks to $X$. We prove bijectivity of the map:
\[
\Hom_{\cC(Y)}(W_1,W_2)\too
\Hom_{\cC(R,X)^{\Sigma}}\bigl((V_1,\phi_1),(V_2,\phi_2)\bigr).
\]
Since $\Hom_Y(W_1,W_2)$ is a sheaf in the fpqc topology (this does
not use flatness of $W_i\to Y$),
injectivity
goes as in~\ref{theorem:descent_along_quotient}. For the
surjectivity part let $f:(V_1,\phi_1)\to(V_2,\phi_2)$
be a morphism, so $s^*f=t^*f$. Since $\Hom_Y(W_1,W_2)$ is a sheaf,
it is enough to prove that
$\pi_1^*f=\pi_2^*f$. We have commutative diagrams:
\[
\xymatrix@C=20mm{
V_1\times_X R \ar[r]^{j^*\pi_i^*f} \ar[d]
& V_2\times_XR \ar[d] \\
V_1\times_X(X\times_YX) \ar[r]^{\pi_i^*f} & V_2\times_X(X\times_YX)}
\]
for $i=1,2$.
From $s^*f=t^*f$ it follows that $j^*\pi_1^*f=j^*\pi_2^*f$.
Since the left vertical map is the pullback of $j$ along the
flat map $V_1\to X$, it is an epimorphism. It then follows that
$\pi_1^*f=\pi_2^*f$.

In order to show essential surjectivity let
$(V,\phi)\in\cC(R,X)^{\Sigma}$. Let $V_i=\pi_i^*V$ be the pullbacks
of $V\to X$ along the projections $\pi_1,\pi_2:X\times_YX\to X$.
When pulling back $j$ along the flat morphism
$h\colon V_1\to X\times_YX$, it remains an effective epimorphism.
\[
\xymatrix{
h^*(\Sigma\times_{(s,t)} R) \ar@<.5ex>[r] \ar@<-.5ex>[r] \ar[d]
\ar@{}[rd]|{\square}
& j^*V_1 \ar[r]  \ar[d] \ar@{}[rd]|{\square} & V_1 \ar[d]^h \\
\Sigma\times_{(s,t)} R \ar@<.5ex>[r] \ar@<-.5ex>[r]
& R \ar[r]^-j & X\times_Y X.}
\]
The morphism
$j^*V_1=s^* V\stackrel{\phi}{\too} t^*V=j^*V_2\too V_2$
is $h^*(\Sigma\times_{(s,t)} R)$-invariant, so by effectivity
we obtain a unique morphism $\psi:V_1\to V_2$.
Similarly we obtain a unique morphism $\chi:V_2\to V_1$.
We claim that $\psi$ and $\chi$
are inverse isomorphisms. Since $j$ is a uniform epimorphism,
in order to prove that the $X\times_Y X$-morphism $\psi\circ \chi$
is the identity it is enough to do it after pullback along $j$.
In this case it is clear since $j^*\psi=\phi$ and $j^*\chi=\phi^{-1}$.
Similarly we prove that $\chi\circ \psi$ is the identity. One
shows as in the end of the proof
of~\ref{theorem:descent_along_quotient} that the isomorphism
$\psi:\pi_1^*V\isomto\pi_2^*V$ is a descent datum for $V$
with respect to $\pi:X\to Y$. The assumptions of the theorem
imply that effective descent along $\pi$ holds in $\cC$ so
$V$ descends to a unique flat morphism $W\to Y$.
\end{proof}

\subsection{Quotient by a subgroupoid} \label{section:quotient_by_gpd}

In this section we come to the quotient question from the introduction,
i.e., the construction of a quotient of a groupoid by a normal
subgroupoid. Let us first review some known cases where
this construction is possible.
\begin{trivlist}
\itemn{1} If $R\rightrightarrows X$ is given by the action of a group space $G$
and $P\rightrightarrows X$ is given by a flat normal subgroup $H$. In this case
the quotient groupoid $Q\rightrightarrows Y$ is the action of $G/H$ on $X/H$.
More generally the quotient exists when $R\rightrightarrows X$ is a local group
action (i.e., it is given by a group action, fppf locally on $X/R$) and $P$
is a flat local normal subgroup action.
\itemn{2} If $R\rightrightarrows X$ is finite locally free and $P$ is a normal
open and closed subgroupoid; this is the Bootstrap Theorem of
\cite[7.8]{KM97}.
\itemn{3} If $P$ is included in the stabilizer; this is the process
of rigidification of~\cite[\S5.1]{ACV03} and~\cite[\S A]{AOV08}.
\end{trivlist}
With suitable flatness assumptions, we shall provide
another case in a different direction: the quotient
exists when $P$ has complexity at most 1.
We emphasize that the existence of the quotients $Y=X/P$ and
$Q=P\backslash R/P$ appearing in the statement is granted
by~\ref{theo:quotient_by_ff_groupoid}.

\begin{theorem} \label{theorem:quotient_by_subgroupoid}
Let $R\rightrightarrows X$ be a flat, locally finitely presented groupoid
of algebraic spaces.
Let $P\rightrightarrows X$ be a flat, locally finitely presented normal
subgroupoid of $R$ with finite stabilizer $\Sigma_P\to X$
and complexity at most $1$. Assume that the quotient
$X\to Y=X/P$ is flat and locally finitely presented. Then there is a quotient
groupoid $Q\rightrightarrows Y$ which is flat and locally finitely presented,
with $Q=P\backslash R/P$. Moreover, the morphisms
$R\to Q$ and $R\times_X R\to Q\times_YQ$ are flat and locally finitely presented.
\end{theorem}

The rest of the text is devoted to the proof.
We denote by $s,t:R\rightrightarrows X$ and $\sigma,\tau:P\rightrightarrows X$
the source and target maps of the groupoids, and by
$\rho:R\to Q$ and $\pi:X\to Y$ the quotient maps.

\bigskip

\noindent {\bf Step 1.} There exist flat
locally finitely presented maps $\bar s,\bar t:Q\rightrightarrows Y$ and
commutative squares:
\[
\xymatrix@C=10mm{R \ar[r]^-{s,t} \ar[d]_\rho & X \ar[d]^\pi \\
Q \ar@{..>}[r]^-{\bar s,\bar t} & Y}
\]
and $\rho$ is flat.
To prove this we start with the action of $P$ on $R$ by
postcomposition. This action is free so there is a flat,
locally finitely presented quotient morphism
$\rho_{\post}:R\to P\backslash R$ where $P\backslash R$ is
an algebraic space. Since $s:R\to X$ is invariant by the
action of $P$,
there is an induced faithfully flat locally finitely presented morphism
$s':P\backslash R\to X$. The map $R\times_{(s,\sigma)} P\to R$,
$(\alpha,\varphi)\mapsto\alpha\varphi^{-1}$ is equivariant for
the action of $P$ on the $R$-factors by postcomposition.
Using that the formation of the quotient
$\rho_{\post}:R\to P\backslash R$ commutes with the flat base
change $\sigma:P\to X$, we deduce that there is an induced map
$(P\backslash R)\times_{(s',\sigma)} P\to P\backslash R$.
In this way we obtain a $P$-linearization on the $X$-object
$P\backslash R$, as follows:
\[
\setlength{\arraycolsep}{1mm}
\begin{array}{rcl}
\sigma^*(P\backslash R)=(P\backslash R)\times_{(s',\sigma)}P
 & \isomto
& \tau^*(P\backslash R)=(P\backslash R)\times_{(s',\tau)}P . \\
(\alpha,\varphi) & \longmapsto & (\alpha\varphi^{-1},\varphi)
\end{array}
\]
We claim that because~$P$ is normal, the restriction of this
$P$-linearization to the stabilizer $\Sigma_P$ is trivial. In
order to check this, we take advantage of the fact that
the space $P\backslash R$ is equal to the fppf quotient sheaf
so locally $(P\backslash R)(T)=P(T)\backslash R(T)$. If
$\varphi\in\Sigma_P(T)$ and $\alpha\in R(T)$, we have
$\psi:=\alpha\varphi^{-1}\alpha^{-1}\in\Sigma_P(T)$ and hence
$\alpha\varphi^{-1}=\psi\alpha$ in $R(T)$ which is equal to $\alpha$
in $(P\backslash R)(T)$. This proves our claim.
It follows from case (1) of
Theorem~\ref{theorem:descent_of_flat} that
$s':P\backslash R\to X$ descends to a faithfully flat
locally finitely presented map $\bar s:Q_1\to Y$.

Similarly, considering the action of $P$ on $R$ by precomposition,
we obtain a flat, locally finitely presented quotient morphism
$\rho_{\pre}:R\to R/P$, and a flat locally finitely presented
morphism $t':R/P\to X$ induced by~$t$. The latter supports a
$P$-linearization with trivial stabilizer action and descends
to a faithfully flat locally finitely presented map $\bar t:Q_2\to Y$.

Since the formation of the quotient $X\to Y$ commutes with
flat base change, we see that $Q_1$ is the quotient of
$P\backslash R$ by $P$ acting by postcomposition and that $Q_2$
is the quotient of $R/P$ by $P$ acting by precomposition.
Both quotients are isomorphic since they enjoy the same
universal property as $Q=P\backslash R/P$. So $Q=Q_1=Q_2$
canonically and we obtain maps $\bar s,\bar t:Q\rightrightarrows Y$.
In this way we obtain also that $\rho:R\to Q$ is flat,
being the composition of the flat map
$\rho_{\post}:R\to P\backslash R$ and of the morphism
$P\backslash R\to Q$ which is a base change of the flat map
$X\to Y$. We have thus produced the commutative diagrams
\[
\xymatrix{
R\ar[r]^-{\rho_{\post}}\ar[rd]_{\rho}
  & P\backslash R \ar[r]^-{s'} \ar[d]^{\pi\smash{'}} \ar@{}[rd]|{\square}
  & X \ar[d]^{\pi} \\
  & Q \ar[r]_{\bar s} & Y}\quad\quad
\xymatrix{
R\ar[r]^-{\rho_{\pre}}\ar[rd]_{\rho}
  & R/P \ar[r]^-{t'} \ar[d]^{\pi\smash{'}} \ar@{}[rd]|{\square}
  & X \ar[d]^{\pi} \\
  & Q \ar[r]_{\bar t} & Y}
\]
in which all maps are flat.

\bigskip

\noindent {\bf Step 2.} There exists a flat locally finitely
presented map $\bar c:Q\times_Y Q\to Q$ and a commutative
square:
\[
\xymatrix@C=10mm{R\times_X R \ar[r]^-c
\ar[d]_{\rho\times\rho} & R \ar[d]^\rho \\
Q\times_Y Q \ar@{..>}[r]^-{\bar c} & Q}
\]
where $\rho\times \rho$ is flat.
To prove this, note that there are three commuting actions of $R$ on $R\times_X
R$: pre-composition $(\alpha,\beta,\gamma): (\alpha,\beta)\to
(\alpha,\beta\gamma)$, post-composition $(\gamma,\alpha,\beta):
(\alpha,\beta)\to (\gamma\alpha,\beta)$ and middle-composition
$(\alpha,\gamma,\beta):(\alpha,\beta)\to (\alpha\gamma,\gamma^{-1}\beta)$.
The joint action of any two of these are free. The composition $c$ is
equivariant with pre- and post-composition and invariant under
middle-composition.

Taking the quotient by post-composition under $P$, we obtain a flat morphism
$c':(P\backslash R)\times_{(s',t)} R\to P\backslash R$. Since $s':P\backslash
R\to X$ is the pull-back of $\bar{s}:Q\to Y$, we can identify the source of
$c'$ with $Q\times_{\bar{s},\pi t} R$. Middle-composition then becomes
post-composition on the second factor so $c'$ factors into two flat maps
\[
\xymatrix{Q\times_{\bar{s},\pi t} R\ar@/^1.0pc/[rr]^{c'}\ar[r]
 & Q\times_{\bar{s},\bar{t}\pi'} P\backslash R\ar[r]_-{c''}
 & P\backslash R.}
\]
The map $c''$ fits into the diagram
\[
\xymatrix{Q\times_{\bar{s},\bar{t}\pi'} P\backslash R\ar@/^1.4pc/[rr]^{s'\pr_2}\ar[r]^-{c''}\ar[d]
 & P\backslash R\ar[r]^{s'}\ar[d]\ar@{}[rd]|{\square} & X\ar[d]^\pi\\
Q\times_Y Q\ar@/_1.0pc/[rr]_{\bar{s}\pr_2}\ar@{..>}[r]^{\bar{c}}
 & Q\ar[r]^-{\bar{s}} & Y}
\]
where the outer square also is cartesian, so $c''$ descends to a flat map
$\bar{c}$ as indicated in the diagram (Theorem \ref{theorem:descent_of_flat}).
The map $\rho\times\rho:R\times_X R\to Q\times_Y Q$ is flat, being the
composition of the flat map $R\times_X R\to Q\times_{\bar{s},\bar{t}\pi'}
P\backslash R$ (quotient map of the free middle-post-composition) and the
pull-back of the flat map $\pi$.

\bigskip

\noindent {\bf Step 3.} Conclusion. 
It is easy to construct the maps $\bar e:Y\to Q$ and
$\bar\imath:Q\to Q$ fitting in commutative diagrams:
\[
\xymatrix@C=10mm{X \ar[r]^e \ar[d]_\pi & R \ar[d]^\rho \\
Y \ar@{..>}[r]^-{\bar e} & Q}
\qquad\qquad
\xymatrix@C=10mm{R \ar[r]^i \ar[d]_\rho & R \ar[d]^\rho \\
Q \ar@{..>}[r]^-{\bar\imath} & Q.}
\]
We skip the details.
From the fact that $\rho:R\to Q$ and
$\rho\times\rho:R\times_X R\to Q\times_Y Q$ are epimorphisms
of algebraic spaces, it follows formally that the maps
$\bar s,\bar t,\bar c,\bar e,\bar\imath$ are
unique, that they give $Q\rightrightarrows Y$ the structure of a groupoid,
and that the map $\rho:R\to Q$ is a morphism of groupoids.
Finally we can prove that the groupoid
$Q\rightrightarrows Y$ is a quotient of $R\rightrightarrows X$ by $P$.
Let $f:(R,X) \to (R',X')$ be a morphism of groupoids such that
$P\subset \ker(f)$. Then the map $f:R\to R'$ is invariant by the
pre-post-composition of~$P$ on $R$, hence it factors through
a map $Q\to R'$. Similarly the map $f:X\to X'$ is invariant by the
action of $P$ hence it factors through a map $Y\to X'$.
That $(Q,Y)\to (R',X')$ is a morphism of groupoids follows
again from the fact that $\rho$ and $\rho\times\rho$ are
epimorphisms.

\subsection{Stacky interpretations}
Let $R\rightrightarrows X$ be a flat locally finitely presented groupoid in
algebraic spaces and let $\cC\to\AlgSp$ be a stack for the fppf topology. Then
the category $\cC(R,X)$ of $R$-equivariant object is equivalent with the
category of morphisms $[X/R]\to \cC$ between stacks. A morphism
$\varphi:[X/R]\to \cC$ corresponds to an object with trivial $\Sigma$-action if
and only if the following equivalent conditions hold
\begin{trivlist}
\itemn{1} For every algebraic space $T$, object $x\in [X/R](T)$, and
automorphism $\tau\in \Aut(x)$, the image $\varphi(\tau)$ is the identity
on $\varphi(x)$.
\itemn{2} The induced morphism of inertia stacks $I\varphi: I[X/R]\to I\cC$ is
trivial, i.e., factors through $\cC$.
\itemn{3} The morphism $\varphi$ factors, up to equivalence, through
the fppf-sheafification $[X/R]\to \pi_0([X/R])$.
\end{trivlist}
If $R\rightrightarrows X$ has finite inertia, then the coarse space $[X/R]\to
X/R$ factors through the fppf sheaf quotient $\pi_0[X/R]=(X/R)_{\fppf}$ and 
$\pi_0[X/R]\to X/R$ is
an isomorphism if the action is free.
Theorem~\ref{theorem:descent_along_quotient} thus says that the functor
\[
\Hom(X/R,\cC)\to \Hom(\pi_0[X/R],\cC)
\]
is an equivalence of categories if $R\rightrightarrows X$ has complexity
at most 1 and under certain assumptions on $\cC$, e.g., if
$\cC$ is a stack in groupoids with representable diagonal.

In the setting of Theorem~\ref{theorem:descent_of_flat}, the category
$\cC(R,X)$ is equivalent to the category of flat morphisms of algebraic
stacks $\cX'\to \cX=[X/R]$ that are representable by algebraic spaces.
The subcategory $\cC(R,X)^\Sigma$ consists of stabilizer-preserving
morphisms, i.e., those such that the induced morphism of inertia stacks
$I\cX'\to (I\cX)\times_{\cX} \cX'$ is an
isomorphism. Theorem~\ref{theorem:descent_of_flat} thus says that the category
of flat morphisms $Y'\to Y=X/R$ is equivalent to the category of flat
stabilizer-preserving representable morphisms of algebraic stacks $\cX'\to
\cX$.

\begin{remark}
It can be proved that case (2) of Theorem~\ref{theorem:descent_of_flat} holds
for arbitrary flat morphisms $X'\to X$. Indeed, let $\cX'\to
\cX$ be the corresponding stabilizer-preserving representable morphism of algebraic stacks. Then $\cX'$ also has finite stabilizer and a coarse moduli space
$Y'=X'/R'$. It is enough to show that the diagram
\[
\xymatrix{\cX'\ar[r]\ar[d] & \cX\ar[d]\\
Y'\ar[r] & Y}
\]
is cartesian. This can be checked \'etale-locally on $Y'$ and $Y$, so we can
assume that $Y$ and $Y'$ are affine. After further \'etale localization on $Y$,
we can assume that $\cX=[X/R]$ where $X\to \cX$ is finite: this follows from
the proof of the Keel--Mori theorem~\cite[Thm.~6.12]{Ry13}. Since $\cX'\to \cX$
is representable, we obtain a finite presentation $X'\to \cX'$ where
$X'=X\times_{\cX} \cX'$. It follows that $X'$ and $X$ are affine
since $X'\to Y'$ and $X\to Y$ are affine~\cite[Thm.~5.3]{Ry13}. Thus $X'\to X$
is affine and case (2) of Theorem~\ref{theorem:descent_of_flat} applies.
\end{remark}

Finally, Theorem~\ref{theorem:quotient_by_subgroupoid} can be described as
follows using stacks. We have a locally finitely presented flat morphism
$[X/P]\to [X/R]$. This gives rise to a groupoid
\[
\xymatrix{%
[X/P]\times_{[X/R]} [X/P] \ar@<.5ex>[r] \ar@<-.5ex>[r] & [X/P]
}%
\]
with quotient $[X/R]$. That $P\subset R$ is a normal subgroupoid implies that the
morphisms of the groupoid above are stabilizer-preserving. We can also make the
identification $[X/P]\times_{[X/R]} [X/P]=[P\backslash R/P]$. By
Theorem~\ref{theorem:descent_of_flat}, we thus obtain a cartesian diagram
\[
\xymatrix{%
[P\backslash R/P] \ar@<.5ex>[r] \ar@<-.5ex>[r] \ar[d] \ar@{}[rd]|{\square}
  & [X/P]\ar[r]\ar[d]\ar@{}[rd]|{\square} & [X/R] \ar[d] \\
\llap{Q=}P\backslash R/P \ar@<.5ex>[r] \ar@<-.5ex>[r] & X/P\ar[r] & [(X/P)/Q]
}
\]
where the horizontal morphisms are flat and locally of finite
presentation and the vertical morphisms are (relative) coarse moduli spaces.

\subsection{A non-flat counter-example}\label{non-flat-example}
We give an example that shows that
Theorems~\ref{theorem:descent_along_quotient} and~\ref{theorem:descent_of_flat}
do not hold when $\pi:X\to Y$ is not flat. The counter-example satisfies
\begin{trivlist}
\itemn{1} $X$ is an affine $1$-dimensional scheme in characteristic $p$ with an
  action of $G=\ZZ/p\ZZ$ but $\pi:X\to Y=X/G$ is not flat.
\itemn{2} There is a torsion equivariant line bundle $\sL\in
  \Pic^G_p(X)^{\Sigma}$ that does not come from $\Pic(Y)$. In particular,
  Theorem~\ref{theorem:descent_along_quotient} fails for the algebraic stacks
  $\cC=\Pic=B\GG_m$ and $\cC=\Pic_p=B\mu_p$.
\itemn{3} There is a smooth morphism $X'\to X$ that is not the pull-back
  of a smooth morphism $Y'\to Y$. In particular,
  Theorem~\ref{theorem:descent_of_flat} fails even for smooth morphisms.
\end{trivlist}

Let $k$ be a field of characteristic $p$. Let $X=\Spec k[\epsilon,x]/(\epsilon^2)$ and let $\ZZ/p\ZZ$ act via
$(t,\epsilon,x)=(\epsilon,x+t\epsilon)$. Then $Y=\Spec k[\epsilon,x^p,\epsilon
  x, \epsilon x^2,\dots,\epsilon x^{p-1}]$.

Consider the following $\ZZ/p\ZZ$-equivariant line bundle $\sL$ on $X$: as a
line bundle it is trivial $\sL=\sO_X\cdot e$ and it has the action
$(t,e)=(1+t\epsilon)e$.

The stabilizer acts trivially on this line bundle. Indeed, the stabilizer
$\Sigma$ of $X$ is given by the closed subscheme $t\epsilon=0$ of
$(\ZZ/p\ZZ)\times X = \Spec k[t,\epsilon,x]/(t^p-t,\epsilon^2)$.

The line bundle is not in the image of $\pi^*:\QCoh(Y)\to \QCoh^G(X)$. Indeed,
since $\pi^*$ has the right adjoint $(\pi_*-)^G$, it is enough to verify that
the counit $\pi^*(\pi_*\sL)^G\to \sL$ is not an isomorphism. But an easy
calculation gives that $\pi^*(\pi_*\sL)^G = (\epsilon)\cdot \sL\subsetneq \sL$.

In terms of algebraic stacks, the line bundle $\sL$ corresponds to the morphism
\[
[X/(\ZZ/p\ZZ)]\to B(\ZZ/p\ZZ)_S\xrightarrow{B\varphi} B\mu_p\to
B\GG_m,
\]
where $S=\Spec k[\epsilon]/(\epsilon^2)$ and $\varphi\colon
(\ZZ/p\ZZ)_S\to \mu_p$ is the group homomorphism given by $t\mapsto
(1+\epsilon)^t=1+t\epsilon$. Here the map between inertia stacks
$I[X/(\ZZ/p\ZZ)]\to IB(\ZZ/p\ZZ)_S\to IB\mu_p$ is induced by
\[
\begin{array}{ccccc}
k[\lambda]/(\lambda^p-1) &\longrightarrow& k[\epsilon,t]/(\epsilon^2,t^p-t) &\longrightarrow& k[\epsilon,x,t]/(\epsilon^2, t^p-t,t\epsilon)\\
\lambda &\longmapsto& 1+t\epsilon &\longmapsto& 1
\end{array}
\]
so it factors through $B\mu_p$.

The line bundle corresponds to the smooth stabilizer-preserving $G$-equivariant
morphism $X'=\Spec k[\epsilon,x,y]/(\epsilon^2)$ where the $G$-action is
$(t,\epsilon,x,y)\mapsto (\epsilon,x+t\epsilon,y+t\epsilon y)$. This is not the
pull-back of the morphism $Y'=X'/G\to Y=X/G$. Indeed, a similar calculation as
for the line bundle gives that $Y'=\Spec k[\epsilon,x^p,y^p,\epsilon x^iy^j]$.
% x^iy^j \mapsto (x^i+i\epsilon x^{i-1})(y^j+j\epsilon y^j)
% \sum a_{ij} x^iy^j \mapsto \sum b_{ij} x^iy^j
% index over i,j=0,...,p-1 and a_{ij} \in k[\epsilon,x^p,y^p]
%
% b_{ij} - a_{ij} = j\epsilon a_{ij}+(i+1)\epsilon a_{(i+1)j}
% For every j, we get a_{ (p-1)j }, a_{ (p-2)j }, ..., a_{ 1j }, a_{ 0j } are in (\epsilon) (except a_00).

\bigskip

\noindent
Matthieu Romagny, \url{matthieu.romagny@univ-rennes1.fr} \\
IRMAR, Universit\'e Rennes 1, Campus de Beaulieu, 35042 Rennes Cedex, France

\bigskip

\noindent
David Rydh, \url{dary@math.kth.se} \\
KTH Royal Institute of Technology, Department of Mathematics, 10044 Stockholm, Sweden

\bigskip

\noindent
Gabriel Zalamansky, \url{g.s.zalamansky@umail.leidenuniv.nl} \\
Universiteit Leiden, Snellius Building, Niels Bohrweg 1, 2333 CA Leiden, The Netherlands

\end{document}